\documentclass{amsart}
\usepackage[english]{babel}
\usepackage{tikz-cd}
\usepackage{inputenc}
\usepackage[T1]{fontenc}
\usepackage{amsmath}
\usepackage{amsthm}
\usepackage{amsfonts}
\usepackage{amssymb}
\usepackage{enumerate}
\usepackage{float}
\usepackage{oubraces}
\usepackage{mathrsfs}
\usepackage{graphicx}
\usepackage{hyperref}
\usepackage{tikz}
\usepackage{braids}
\usepackage{amscd,amsmath,amssymb,amsthm,amsfonts,epsfig,graphics}
\begin{document}

\setlength{\unitlength}{1 cm} %Especificar unidad de trabajo

\pagenumbering{arabic}

\newtheorem{teo}{Theorem}[section]
\newtheorem{mydef}[teo]{Definition}
\newtheorem{ex}[teo]{Example}
\newtheorem{prop}[teo]{Proposition}
\newtheorem{lem}[teo]{Lemma}
\newtheorem{corol}[teo]{Corollary}
\newtheorem{case}{Case}
\newtheorem{obs}[teo]{Remark}
\newtheorem{conj}[teo]{Conjecture}
\newtheorem{notat}[teo]{Notation}
\newtheorem{thmx}{Theorem}
\renewcommand{\thethmx}{\Alph{thmx}}

\vspace{0.5 cm}

\title{Poly-freeness in large even Artin groups}
\author[R. Blasco]{Rub\'en Blasco-Garc{\'\i}a}
\address{Departamento de Matem\'aticas, IUMA\\ 
	Universidad de Zaragoza\\ 
	C.~Pedro Cerbuna 12\\ 
	50009 Zaragoza, Spain} 
\email{rubenb@unizar.es}

\maketitle

\begin{abstract}
	We prove that any large even Artin group is poly-free and that any even Artin group based on a triangle graph is also poly-free.
\end{abstract}

{Primary: 20F36}

\section{Introduction}

The main objective of this paper is to prove that large even Artin groups are poly-free. 

Artin groups (also called Artin-Tits groups or generalized braid groups) form one of the families of groups where the use of both algebraic and geometric techniques has been most successful. Artin groups can be defined in the following way. Let $\Gamma$ be a simple labeled graph (a graph with no multiple edges between any pair of vertices and no edge from one vertex to itself) with a finite set of vertices $V$ and a set of edges $E$ such that for every edge $e\in E$ there is a label $2\leq m_e\in\mathbb{Z}$. The \textbf{Artin group} associated to $\Gamma$, $A_{\Gamma}$, is the group with the following presentation:
\begin{equation}
\label{eq-Artingroup}
A_{\Gamma}=\langle v; v\in V\mid {}_{m_e}(uv)=
{}_{m_e}(vu), e=\{u,v\}\in E \rangle.
\end{equation}
where $_{m_e}(uv)$ denotes the alternating product $uvu \cdots$ of length $m_e$.

There are few results known to be true for the whole family of Artin groups and usually one considers more or less general subfamilies. One of the most studied subfamilies is the family of \textbf{right-angled Artin groups} which are those Artin groups that satisfy that $m_e=2$ for every $e\in E$. For an introduction to the family of right-angled Artin groups see \cite{Charney3}.

An Artin group is called \textbf{even} if $m_e$ is an even number for all $e\in E$. There are not many results in the literature about even Artin groups (but see \cite{Blasco1}, \cite{Blasco2} and \cite{Blasco3}). However, we think that this family deserves more attention since they have remarkable properties some of them shared with the familiy of right-angled Artin groups. For example, even Artin groups retract onto any parabolic subgroup. In this paper we will take advantage of that and also of the fact that Artin relations $ {}_{m_e}(uv)=
{}_{m_e}(vu)$ of even type, i.e. when $m_e$ is an even number, can be rewritten in terms of conjugation (see section \ref{S3}).

An Artin group is said to be \textbf{large} if $m_e\geq 3$ for every $e\in E$. This family has nice properties, for example the word problem is solvable for large Artin groups. Holt and Rees described a set of normal forms for large Artin groups \cite{Holt}, these normal forms will play a key role in this paper.

A group $G$ is said to be \textbf{poly-free} if there exists a tower of normal subgroups

\begin{center}
	$1=G_{0}\unlhd G_{1} \unlhd ... \unlhd G_{N}=G $
	\end{center}
such that every quotient $G_{i+1}/G_{i}$ is free. This property implies other interesting algebraic properties. For example, any poly-free group is locally indicable (i.e. every non-trivial finitely generated subgroup admits an epimorphism onto $\mathbb{Z}$), and locally indicable groups are right-orderable (i.e. admit a total order which is invariant under right multiplication), see Rhemtulla-Rolfsen \cite{Rhemtulla}.
	
	Some families of Artin groups are known to be poly-free. For example, right-angled Artin groups are poly-free: this fact was independently proved by Duchamp and Krob \cite{Duchamp}, Howie \cite{Howie} and Hermiller and {\v{S}}uni{\'c} \cite{Hermiller}. Mart\'inez-P\'erez, Paris and the author proved that even Artin groups of type FC are also poly-free \cite{Blasco1}. In this paper, we add to the list of Artin groups known to be poly-free the family of large even Artin groups. Our main theorems:
	
	\begin{thmx}\label{TA}
	Any large even Artin group is poly-free.
	\end{thmx}
	
	\begin{thmx}\label{TB}
		Any even Artin group associated to a triangle graph is poly-free.
	\end{thmx}

	In section \ref{S2} we review some results by Holt and Rees about normal forms in large Artin groups. Section \ref{Section3} is rather technical: we use the normal forms of section \ref{S2} together with other results from \cite{Holt} to gain information about geodesic words in large even Artin groups. 
	
	In section \ref{S3} we will see how to split any large even Artin group as a semidirect product of a parabolic subgroup and a particular normal subgroup. Later on we will show that this normal subgroup is free and this semidirect product decomposition will be crucial to argue by induction and deduce our main result. Finally, in section \ref{sec} we will prove Theorem \ref{TA}  and in section \ref{S6} we will prove Theorem \ref{TB}

	\subsection*{Acknowledgements} {The author would like to thank Conchita Mart\'inez-P\'erez for her comments and suggestions.
		
	The author was partially supported by a Departamento de Industria e Innovaci\'on del Gobierno de Arag\'on and Fondo Social Europeo PhD grant, the Spanish Government MTM2015-67781-P (MINECO/FEDER) and MTM2016-76868-C2-2-P and Grupo Algebra y Geometr\'ia from Gobierno de Arag\'on.}

\section{Normal forms in Artin groups}\label{S2}

In this section we will recall some definitions and results about normal forms which can be found in \cite{Holt}, \cite{Brien}, \cite{Holt2} applying them to the particular case of even Artin groups.

\begin{mydef}
	
	We define an \textbf{alphabet} to be a finite set $L$. An element $a\in L$ is called a \textbf{letter}. A \textbf{word} over $L$ is a finite sequence of letters.
	
	Formally, a word can be defined as a map $w : \{1,...,n\} \rightarrow L$ where $w (i)$ is the i-th letter of the word. The \textbf{length} of a word $w$ is the integer $n$ and it is denoted by $|w|$.

\end{mydef}

When $n=0$, we say that $w$ is the \textbf{empty word} over $L$, and it is denoted by $\epsilon$. We denote by $L^{*}$ the set of all words over the alphabet $L$. Given a word $w=abc$, with possibly empty $a,b,c\in L^{*}$, the word $a$ is said to be a \textbf{prefix} of $w$, $c$ is a \textbf{suffix} of $w$ and $b$ is a \textbf{subword} of $w$. Given a word $w$ we denote by $f[w]$ and $l[w]$ the first and last letter of $w$ respectively. So, if $|w|=n$, $f[w]=w(1)$ and $l[w]=w(n)$.

From now on we fix $L=S\cup S^{-1}$, where $S$ is a generating set of the group $G$. A letter $a\in L$ is \textbf{positive} if $a\in S$ and is \textbf{negative} otherwise. The \textbf{name} of a letter is its positive form.

If two words $w, v$ represent the same element in a group $G$, we will write $w =_{G} v$.

\begin{mydef}
	
	We say that a word $w\in L^{*}$ is \textbf{positive} if all its letters are positive, \textbf{negative} if all its letters are negative and \textbf{unsigned} otherwise.
	
\end{mydef}

\begin{mydef}
	
	A word $w\in L^{*}$ is \textbf{freely reduced} if it does not admit any subword of the form $aa^{-1}$ or $a^{-1}a$ for any letter $a$. We say that a not freely reduced word \textbf{admits a free reduction}.
	
	A word $w\in L^{*}$ is \textbf{geodesic} if for any other word $v$ such that $w =_{G} v$, we have that $|w|\le |v|$.
	
\end{mydef}

\begin{mydef}
	
	Let $L$ be an alphabet. Given $<_{lex}$ an arbitrary lexicographic ordering on $L$ (which induces a lexicographic order, that we denote $<_{\text{lex}}$, on $L^*$), the \textbf{shortlex} ordering $<_{slex}$ on $L^{*}$ is defined by
	
	\begin{center}
		$w< _{slex} v \text{ if and only if } |w|<|v|  \text{ or } |w|=|v| \text{ and } w<_{lex} v$.
	\end{center}
	
\end{mydef}

\begin{mydef}
	
	A word $w$ is said to be a \textbf{shortlex minimal representative} if for every word $v\neq w$ such that $w=_{G} v$, $w<_{slex} v$.
	
\end{mydef}

\begin{obs}
	
	Every shortlex minimal representative is geodesic.
	
\end{obs}

\subsection{Dihedral Artin groups}

\begin{mydef}
	
	The \textbf{dihedral Artin group} $A_{2}(m)$, $m\in\mathbb{Z}^+\cup \{+\infty\}$ is the Artin group based on the graph consisting of two vertices joined by an edge labeled with $m$ or two disconnected vertices if $m=\infty$. If $m<\infty$ this is the group with presentation $\langle a,b \mid {}_{m}(a,b)={}_{m}(b,a)\rangle$.
	
\end{mydef}

 We want to study how to obtain a shortlex representative of a given word in $A_2(m)$ for $m$ arbitrary. If $m=\infty$, $A_{2}(\infty)$ is the free group on two variables and it is easy to see that every freely reduced word $w\in A_{2}(\infty)$ is shortlex minimal. So we only need to consider the cases where $m<\infty$.

We already defined $_{m}(a,b)$ (on the first page). Similarly, we define $(a,b)_{m}$ as the alternating procuct of $a$ and $b$ of length $m$ ending with $b$. Notice that if $m$ is even, we have $_{m}(a,b)=(a,b)_{m}$ and we can use the two expressions interchangeably.

\begin{mydef}
	
	Let $w$ be a freely reduced word in $A_{2}(m)$ over the alphabet $L=\{a, a^{-1}, b, b^{-1}\}$. Consider the integers:
	
	\begin{center}
		$r_{1}=\max\{r\mid$ $_{r}(a,b) \text{ or } _{r}(b,a) \text{ is a subword of } w \},$
		
	\end{center}
	
	\begin{center}
		
		$r_{2}=\max\{r\mid$ $_{r}(a^{-1},b ^{-1}) \text{ or } _{r}(b^{-1},a^{-1}) \text{ is a subword of } w \},$
	\end{center}
	
	$p(w)=\min\{r_{1},m\}$ and $n(w)=\min\{r_{2},m\}$.
	
\end{mydef}

Geodesic words $w$ in $A_{2}(m)$ are characterized by the values $p(w)$ and $n(w)$.

\begin{prop}\cite{Mairesse} Let $g\in A_{2}(m)$ and let $w\in L^{*}$ be a freely reduced word representing $g$.
	
	\begin{enumerate}
		\item If $p(w)+n(w)<m$, then $w$ is the unique geodesic representative for $g$.
		\item If $p(w)+n(w)=m$, then $w$ is one of the geodesic representatives for $g$.
		\item If $p(w)+n(w)>m$, then $w$ is not geodesic. Furthermore, $w$ has a subword $w'$ such that $p(w')+n(w')=m$.
		
	\end{enumerate}
\end{prop}

Let $w$ be a shortlex minimal word representing $g$. Then since $w$ is geodesic, $p(w)+n(w)\leq m$. Moreover, if $p(w)+n(w)<m$, then $w$ is the unique geodesic word for $g$ and such word must be the shortlex minimal representative of $g$.

\begin{mydef}\label{critical words}\cite{Holt}
	Let $w$ be a freely reduced word in $A_{2}(m)$. Let $\{x,y\}=\{z,t\}=\{a,b\}$ and put $p=p(w)$ and $n=n(w)$. The word $w$ is called a \textbf{critical word} if $p+n=m$ and it has one of the following forms. In these forms, $\xi_+$ represents some positive word in $L^{*}$, $\xi_-$ some negative word in $L^{*}$ and $\eta$ some word in $L^{*}$.
	
	If $w$ is a positive word, then
	
	\begin{center}
		$w=\xi_+(x,y)_{m} \text{ or } w=$ $_{m}(x,y)\xi_+$,
	\end{center}
	where $w$ has exactly one alternating positive subword of length $m$.
	
	If $w$ is a negative word, then
	
	\begin{center}
		$w=\xi_-(x^{-1},y^{-1})_{m} \text{ or } w=$ $_{m}(x^{-1},y^{-1})\xi_-$,
	\end{center}
	where $w$ has exactly one alternating negative subword of length $m$.
	
	If $w$ is an unsigned word, then
	
	\begin{center}
		$w=$ $_{p}(x,y)\eta(z^{-1},t^{-1})_{n} \text{ or } w=$ $_{n}(x^{-1},y^{-1})\eta(z,t)_{p}$.
	\end{center}
	
	We denote by $T$ the set of all critical words.
	
\end{mydef}

Let us consider the Garside element $\Delta:= (a,b)_m$. Notice that $\Delta$ is central if $m$ is even but $a^{\Delta}=b$ if $m$ is odd. We define the automorphism $\nu$ of $L^*$ such that $\nu(w)=w^{\Delta}$. Notice that for the case of even Artin groups $\nu=Id_{L*}$.

\begin{mydef}\label{taumoves}
	
	Define a map $\tau$ on the critical words as follows:
\begin{align*}
		(x,y)_{m} &\mapsto\hspace{0.1cm} (y,x)_{m}, \\
		\xi_+ (x,y)_{m} &\mapsto\hspace{0.1cm} _{m}(t,z)\nu(\xi_+), \text{ where } z=f[\xi_+],  \\
		_{m}(x,y)\xi_+ &\mapsto\hspace{0.1cm} \nu(\xi_+)(z,t)_{m}, \text{ where } z=l[\xi_+], \\
		(x^{-1},y^{-1})_{m} &\mapsto\hspace{0.1cm} (y^{-1},x^{-1})_{m}, \\
		\xi_- (x^{-1},y^{-1})_{m} &\mapsto\hspace{0.1cm} _{m}(t^{-1},z^{-1})\nu(\xi_-), \text{ where } z=f[\xi_-]^{-1}, \\
		_{m}(x^{-1},y^{-1})\xi_- &\mapsto\hspace{0.1cm} \nu(\xi_-) (z^{-1},t^{-1})_{m}, \text{ where } z=l[\xi_-]^{-1}, \\
		_{p}(x,y)\eta(z^{-1},t^{-1})_{n} &\mapsto \hspace{0.1cm}_{n}(y^{-1},x^{-1})\nu(\eta)(t,z)_{p}, \text{ where } p+n=m, \hspace{.2 cm} n,p\neq 0, \\
		_{n}(x^{-1},y^{-1})\eta(z,t)_{p} &\mapsto \hspace{0.1cm}_{p}(y,x)\nu(\eta)(t^{-1},z^{-1})_{n}, \text{ where } p+n=m, \hspace{.2 cm} n,p\neq 0. 
		\end{align*}

	Here $\{x,y\}=\{z,t\}$, $\xi_+$ is a non-empty positive word, $\xi_-$ a non-empty negative word and $\eta$ can be empty.
	
	These are called $\tau$-moves.
	
\end{mydef}

\begin{prop}\label{cons} \cite{Holt}\cite{Brien} For any critical word $w$:
	
	\begin{enumerate}
		\item $\tau(w)$ is also critical, $\tau(w)=_{G} w$ and $\tau(\tau(w))=w$.
		
		\item $p(\tau(w))=p(w)$ and $n(\tau(w))=n(w)$.
		
		\item $f[w]$ and $f[\tau(w)]$ have different names, the same is true for $l[w]$ and $l[\tau(w)]$.
		
		\item $f[w]$ and $f[\tau(w)]$ have the same sign if $w$ is positive or negative, but different signs if $w$ is unsigned; the same is true for $l[w]$ and $l[\tau(w)]$.
	\end{enumerate}
	
\end{prop}

\begin{mydef}\label{critical red}
	
	Let $w$ be a word that admits a factorization $w=w_{1}w_{2}w_{3}$ where $w_{2}$ is a critical word. If $w_{1}\tau(w_{2})w_{3}$ admits a free reduction or if $w_{1}\tau(w_{2})w_{3}<_{lex}w_{1}w_{2}w_{3}$, we say that $w$ admits a \textbf{critical reduction}.
	
	We say that $w_{1}\tau(w_{2})w_{3}$ is a \textbf{right length reduction} (or that $w$ admits a right length reduction) if $l[\tau(w_{2})]=f[w_{3}]^{-1}$.
	
	We say that  $w_{1}\tau(w_{2})w_{3}$ is a \textbf{lex reduction} (or that $w$ admits a lex reduction) if $f[\tau(w_{2})]<_{lex}f[w_{2}]$ and  $w_{1}\tau(w_{2})w_{3}$ does not admit a free reduction. 
	
\end{mydef}

\begin{teo}\label{NF}\cite{Holt}\cite{Brien}
	Let $W$ be the set of all words that do not admit any right length reduction or left lex reduction. Then $W$ is the set of shortlex minimal representatives of elements of $A_{2}(m)$.
\end{teo}

\begin{obs}\label{pref}
	Since $W$ is the set of shortlex minimal representatives, it is clear that if $w\in W$, then for every prefix $u$ of $w$, also $u\in W$.
\end{obs}

\subsection{Large Artin groups}

 To find normal forms for arbitrary large Artin groups we will use the results we had for dihedral groups.
Let $A_{\Gamma}$ be a large even Artin group with $n$ generators $a_{1},...,a_{n}$, let $m_{i,j}$ be the label of the edge between $a_{i}$ and $a_j$, $L=\{a_{1}^{\pm 1},...,a_{n}^{\pm 1}\}$. We fix a lexicographic order on $L$. Given an arbitrary word $w\in L^{*}$, we can consider all the subwords of $w$ involving only $2$ generators, say $a_{i}$ and $a_{j}$, as words in the dihedral group $A_{2}(m_{i,j})$. For such a subword $u$, we define $p(u)$, $n(u)$ and $\tau$ as before. We will denote by $T_{i,j}$ the set of critical words in the subgroup $A_{2}(m_{i,j})$. 

\begin{mydef}\label{overlap}
	Let $w$ be a freely reduced word over $A_{\Gamma}$. Suppose that $w$ has a factorization $\alpha w_{1}uw_{2}\beta$, where $u\in T_{i,j}$.
	
	If $w_{1}$ is a 2-generator subword on $a_{i_{1}},a_{j_{1}}$ such that $|\{i,j\}\cap\{i_{1},j_{1}\}|=1$, the name of $l[w_{1}]$ is not in $\{a_{i},a_{j}\}$ and $w_{1}f[\tau(u)]\in T_{i_{1},j_{1}}$, then we say that we have a \textbf{critical left overlap} in $w_1f[\tau(u)]$.
	
	Similarly, if $w_{2}$ is a 2-generator subword on $a_{i_{2}},a_{j_{2}}$ such that $|\{i,j\}\cap\{i_{2},j_{2}\}|=1$, the name of $f[w_{2}]$ is not in $\{a_{i},a_{j}\}$ and $l[\tau(u)]w_{2}\in T_{i_{2},j_{2}}$, then we say that we have a \textbf{critical right overlap} in $l[\tau(u)]w_2$.
	
\end{mydef}

\begin{mydef}\label{critical seq}
	
	Let $\alpha_{1} u_{1}\beta_{1}$ be a freely reduced word with $u_{1}$ a critical subword. Consider a sequence:
	
	\begin{center}
		$\alpha_{1} u_{1}\beta_{1}, \hspace{.5cm} \alpha_{1} \tau(u_{1})\beta_{1}=\alpha_{2} u_{2}\beta_{2} \hspace{.5cm }...\hspace{.5cm} \alpha_{k}\tau(u_{k})\beta_{k}$.
	\end{center}
	where all the $u_{i}'s$ are critical subwords, $u_i=w_if[\tau(u_{i-1})]^{*}$  for $w_i$ a suffix of $\alpha_{i-1}$ (in this case we have a critical left overlap) or $u_i=l[\tau(u_{i-1})]^{*}w_i$ for $w_i$ a prefix of $\beta_{i-1}$ (in this case we have a critical right overlap), $i=2,...k$, where $*$ represents some positive power $\geq 1$.
	
	 If at each step we have a left (resp. right) critical overlap, we say that this is a \textbf{leftward} (resp. \textbf{rightward}) \textbf{critical sequence}.
	
	If a critical sequence is such that $\alpha_{k}\tau(u_{k})\beta_{k}$ is not freely reduced, the sequence is called a \textbf{length reducing sequence}. In this situation, we say that $\beta_k$ is the \textbf{tail} of the sequence. If it is reduced and  $\alpha_{k}\tau(u_{k})\beta_{k}<_{lex}\alpha_{1} u_{1}\beta{1}$, then it is a \textbf{lex reducing sequence}.

\end{mydef}

Now, we are going to see some examples to illustrate how critical sequences work. For both examples we are going to consider the large even Artin group based on a triangle with labels $(4,4,4)$, $$A_{\Gamma}=\langle a,b,c \mid abab=baba, acac=caca, bcbc=cbcb\rangle. $$

\begin{ex}\normalfont
Let us consider the word $w=cbcabacbcb$ and the lexicographic order $a<a^{-1}<b<b^{-1}<c<c^{-1}$.
	$$w=cbcaba\overbrace{cbcb}^{u_1} \xrightarrow[\tau(u_1)]{} w_1=cbc\overunderbraces{&\br{2}{u_2}}{&aba&b&cbc}{&&\br{2}{\tau(u_1)}}\xrightarrow[\tau(u_2)]{}w_2=\overunderbraces{&\br{2}{u_3}}{&cbc&b&aba&cbc}{&&\br{2}{\tau(u_2)}}\xrightarrow[\tau(u_3)]{} w_3= \underbrace{bcbc}_{\tau(u_3)}abacbc.$$
	
 And since $w_3<_{lex} w$, this is a leftward lex reducing sequence.
\end{ex}

\begin{ex}
Let us consider the word $w=a^{-1}b^3abc^{-1}a^2cb^{-1}aba$ and the lexicographic order $a<a^{-1}<b<b^{-1}<c<c^{-1}$.
	$$w=\overbrace{a^{-1}b^3ab}^{u_1}c^{-1}a^2cb^{-1}aba\xrightarrow[\tau(u_1)]{} w_1=\overunderbraces{&&\br{2}{u_2}}{&bab^3&a^{-1}&c^{-1}a^2c&b^{-1}aba}{&\br{2}{\tau(u_1)}}\xrightarrow[\tau(u_2)]{} w_2=\overunderbraces{&&\br{2}{u_3}}{bab^{3}&ca^2c^{-1}&a^{-1}&b^{-1}ab&a}{&\br{2}{\tau(u_2)}} $$
	
	$$\xrightarrow[\tau(u_3)]{} w_3=\overunderbraces{&&\br{2}{\text{Free red}}}{bab^{3}ca^2c^{-1}&bab^{-1}&a^{-1}&a}{&\br{2}{\tau(u_1)}}\xrightarrow[Free Red]{}w'=bab^{3}ca^2c^{-1}bab^{-1} $$

	 Since $|w'|<|w|$, this is a rightward length reducing sequence.
	
\end{ex}

\begin{teo}\label{NF2}\cite{Holt}
	Let $G$ be a large Artin group. Let $W$ be the set of all freely reduced words $w$ that admit no rightward length reducing sequence or leftward lex reducing sequence of any length $k\ge 1$. Then $W$ is the set of shortlex representatives. 
	
\end{teo}

\begin{prop}\label{Simpl}\cite{Holt}
	Suppose that $w\in W$ and $a\in L$ is such that $wa$ is freely reduced but $wa\not\in W$. Then applying a single rightward length reducing or leftward lex reducing sequence followed by a free reduction in the rightward case (where the letter $a$ will be the tail of the sequence) to $wa$ we get an element of $W$.
	
\end{prop}

	\begin{obs}
		
		Given $w\in L^*$ denote by $sl(w)$ the shortlex minimal representative of $w$. We use the same notation for $g\in G$ and put $sl(g)$ for its shortlex minimal representative.
		
	\end{obs}

\begin{obs}\label{pref2}
	Again, since $W$ is the set of shortlex minimal representatives, it is clear that if $w\in W$, then $u\in W$, for every subword  $u$ of $w$.
\end{obs}

\section{Technical results about geodesic words in large Artin groups}\label{Section3}

In this section we will prove several results concerning geodesic words in large Artin groups that we will later use to prove that groups in this family are polyfree.

\begin{mydef}\label{prefylet}
	Given a word $w$ of length $n$, we will denote by $pr(w,k)$, $k=1,...,n$ the prefix of length $k$ of $w$.
\end{mydef}

Our first technical lemma is equivalent to Proposition 4.5 (1) in \cite{Holt}. However, we will include a proof to introduce some of the techniques that we will use in the rest of the paper.

\begin{lem}\label{Lema 1}
	Let $G$ be a large Artin group, $a\in L$. There are no two geodesic words $aw_1, a^{-1}w_2$ such that $aw_1=_{G}a^{-1}w_2$.
\end{lem}

\begin{proof}
	
	Let us suppose that both words are geodesic and represent the same element of the large Artin group $G$. We may assume that $w_1$ is a shortlex representative. Without loss of generality we will suppose that $a<a^{-1}$ are the first letters in the lexicographic order. Let us define $\hat{w}=a^{-1}w_2$, then we can suppose that $sl(\hat{w})=aw_1$ (it must have this form since $a$ is the first letter of the lexicographic order and the word is geodesic). Let $n=|\hat{w}|$.
	
	It is clear that since $\hat{w}$ is geodesic, $pr(\hat{w},k)$ is also geodesic for every $k=1,...,n$. Obviously, we have $pr(\hat{w},k)=a^{-1}\alpha_k$, $k=1,...,n$, where $\alpha_1=\epsilon$. Since in the lexicographic order before $a^{-1}$ there is only $a$ and $sl(pr(\hat{w},n))=sl(\hat{w})=aw_1$, there exists $m\in\{2,...,n\}$ such that $f[sl(pr(\hat{w},m))]=a$ and $f[sl(pr(\hat{w},m-1))]=a^{-1}$. But since $sl(pr(\hat{w},m-1))\in W$, there is a leftward lex reducing sequence that transforms $sl(pr(\hat{w},m-1))\hat{w}(m)$ into $sl(pr(\hat{w},m))$. Let us define $l$ to the length of that sequence, and let $u_l$ be the last critical word of the sequence.
	
	Since that sequence must change the first letter of the word, we have that $\tau(u_l)$ must begin with $a$ and $u_l$ must begin with $a^{-1}$. But, by Proposition \ref{cons}, $f[\tau(u_l)]$ and $f[u_l]$ must have different names, giving a contradiction.	
\end{proof}

\begin{lem}\cite{CHR}\cite{HRConjugacy}\label{Lema 2}
	Consider a large Artin group $G$ and let $b\in L$. Let $t>0$, $w$ a word. Then:
	\begin{itemize}
		\item A word $b^{t}w$ is geodesic if and only if $bw$ is geodesic.
		
		\item  A word $b^{-t}w$ is geodesic if and only if $b^{-1}w$ is geodesic.
		
	\end{itemize} 
\end{lem}

\begin{proof}
	We are going to prove only the case when the word is $b^{t}w$, the case of $b^{-t}w$ is completely analogous.
	The fact that if $b^{t}w$ is geodesic then $bw$ is geodesic is obvious by Remark \ref{pref2}.
	Let us suppose then that $bw$ is geodesic but $b^tw$ is not geodesic. Thus, there must exist $k>1$ such that $b^{k}w$ is not geodesic but $b^{k-1}w$ is geodesic. Therefore, $w^{-1}b^{-(k-1)}$ is geodesic, but $w^{-1}b^{-k}$ is not geodesic.
	Hence, $sl(w^{-1}b^{-(k-1)})b^{-1}$ is not geodesic and by Theorem \ref{Simpl} there exists a rightward length reducing sequence in which $b^{-1}$ is going to be the tail that will be eliminated. Let us define $l$ to the length of that sequence, and let $u_l$ be the last critical word of the sequence. Therefore, $sl(w^{-1}b^{-(k-1)})=_G \hat{w}\tau(u_l)$ is geodesic and $\tau(u_l)$ ends with $b$. But then, $sl(w^{-1}b^{-(k-1)})$ has a geodesic representative ending with $b$, so $b^{k-1}w$ has a geodesic representative beginning with $b^{-1}$, which is impossible by Lemma \ref{Lema 1}.
\end{proof}

\begin{lem}\label{Lema 3}
	Consider a large Artin group $G$ and let $b,h\in L$,  $w'\in L^*$. If $w=b^t w'$ (resp. $w=b^{-t}w'$) is geodesic but $wh$ doesn't have a geodesic representative beginning with $b^t$ (resp. $b^{-t}$), then $w'h$ has a geodesic representative beginning with $b^{-1}$ (resp. $b$).
\end{lem}

\begin{proof}
	We are going to prove the result only in the case  $w=b^tw'$ with $t>0$, the other case is completely analogous.
	
	The hypothesis implies that $b^tw'h$ is not geodesic, thus by Lemma \ref{Lema 2}, $bw'h$ is not geodesic. Let us consider the word $w_1=bw'$ which is geodesic, then $w_1h=bw'h$ doesn't have a geodesic representative beginning with $b$ (otherwise, $wh$ would have a geodesic representative beginning by $b^t$ contradicting the hypothesis). We have
	$$(w_1h)^{-1}=h^{-1}w_1^{-1}=h^{-1}w'^{-1}b^{-1}=_G sl(h^{-1}w'^{-1})b^{-1}.$$
	
	But since $w_1h$ doesn't have a geodesic representative beginning with $b$, then $sl(h^{-1}w'^{-1})b^{-1}$ is not geodesic and by Theorem \ref{Simpl} there is a rightward length reducing sequence for $sl(h^{-1}w'^{-1})b^{-1}$ such that $b^{-1}$ is the tail and $sl(h^{-1}w'^{-1})=_G \hat{w}\tau(u_l)$ geodesic with $l[\tau(u_l)]=b$. 
	
	Therefore, there exists a geodesic representative of $h^{-1}w'^{-1}$ ending with $b$, and then $w'h$ has a geodesic representative beginning with $b^{-1}$.
\end{proof}

In the next technical results we want to understand when we can have elements admitting geodesic representatives of the form
 $a^sw_1=_G b^t w_2$.

\begin{ex}\normalfont
	Consider the group $A_{\Gamma}=\langle a,b\mid (ab)^2=(ba)^2\rangle$. Note that the relation $abab=baba$ implies $[a,bab]=[b,aba]=1$. From this we get:
	$$b^2(aba)^2=abababab=a(baba)bab=a^2babbab=a^2(bab)^2.$$
\end{ex}

The motivation of the following technical results is precisely to understand better this kind of situation. From now on, let $A_{\Gamma}$ be a large Artin group and let $W$ be the set of words that admit neither rightward length reducing sequence nor leftward lex reducing sequence. Recall that then $W$ is the set of shortlex representatives by Theorem \ref{NF2} (in particular, words in $W$ are geodesic).

Notice that the set $W$ of shortlex representatives depends on the chosen order in the generating set. In the following technical results we only want to prove certain restrictions on the form of the geodesic representatives and to do so, we use appropriate lexicographic orders in each case. This has the effect that we will have to check consistence of the chosen order when applying these results.

\begin{lem}\label{Lema 5}
	Let $G=A_{\Gamma}$ be a large Artin group.
	Suppose that $a$ (resp. $a^{-1}$) is the first letter of the lexicographic order and that there exist a geodesic word $u$ and a letter $l$ such that
	
	\begin{itemize}
		\item  $ul$ is geodesic,
		\item $sl(ul)$ begins with $a^s$, $s>1$, (resp. $s<-1$),
		\item $sl(u)$ doesn't begin with $a^s$.
	\end{itemize}  Then $sl(u)$ must begin with $a^i$ with $i=s-1$ (resp. $i=s+1$).
	
\end{lem}

\begin{proof}
	Without loss of generality we may suppose that $u$ is a shortlex representative.  We are going to argue by induction on $|ul|=k$.
	We will only prove the result for $s>1$, the case $s<-1$ is analogous.
	
	If $|ul|=2$, then $ul=a^2$ and therefore the only option is $u=l=a$, so $s=2$ and $i=s-1=1$.
	
	Let us suppose that the result is true for length smaller than $k$. Recall that $sl(ul)=a^s\hat{u}$ and that $sl(u)=a^i u'$ for some $i\in\mathbb{Z}$, some $\hat{u}$ and some $u'$. We may also assume that $\hat{u}$ and $u'$ don't begin with $a$.
	
	Note that as $sl(u)$ is shortlex and since $a$ is the first letter of the lexicographic order (thus $a<a^{-1}$), $u$ cannot have any geodesic representative beginning by $a^s$.  We want to prove that $i=s-1$ if $s>1$. Note that $u'$ is also shortlex and does not begin with $a^{\pm 1}$ so it doesn't have any geodesic representative beginnig with $a$. Moreover, as $u$ is geodesic, $u'$ cannot have any geodesic representative beginning with $a^{-1}$. Finally notice that applying Lemma \ref{Lema 1} to $sl(u)l$ and $sl(ul)$, we deduce that $i$ and $s$ must have the same sign.

	We know by Proposition \ref{Simpl} that there exists a leftward length reducing sequence that transforms $sl(u)l=a^i u' l$ into $sl(ul)=a^s\hat{u}$.
		$$... \rightarrow \alpha_r u_r \beta_r \rightarrow \alpha_r\tau(u_r)\beta_r=a^s\hat{u}$$
	
	Note that $\alpha_r$ is a prefix of both $sl(u)l=a^i u' l$ and $sl(ul)=a^s\hat{u}$, so the number of $a's$ at the beginning of $\alpha_r$ cannot be bigger than $i$. Since $\alpha_r\tau(u_r)\beta_r$ begins with $a^s$, $\alpha_r=a^p$, $p\leq i$, and as $f(u_r)$ and $f(\tau(u_r))$ must have different names, we conclude that $\alpha_r=a^i$, which implies $a^s\hat{u}=a^i a^{s-i}\hat{u}$ and $sl(u'l)=a^{s-i}\hat{u}$. But now, if $s-i\neq 1$, using the inductive hypothesis, we would have that $sl(u')$ begins with $a$ which is impossible. So, $s-i=1$ and therefore $i=s-1$.
\end{proof}

\begin{corol}\label{Corol 6}
	Suppose that $a$ is the first letter of the lexicographic order and let $w=b^{\pm 1}w'$ be geodesic such that $sl(w)=a^s\tilde{w}$, $s>0$, with $a\neq b$ a letter.
	Then there are prefixes $pr(w,l_1),...,pr(w,l_s)$, $l_1<l_2<...<l_s$, such that $sl(pr(w,l_i))$ begins with $a^i$, but not by $a^{i+1}$.
\end{corol}

\begin{proof}
	We will argue by induction over $s$. 
	If $s=1$ there is nothing to prove. Let us suppose that it is true for $s-1$. It is clear that $pr(w,j)=pr(w,j-1)w(j)$. Let $r$ be the greatest integer such that $sl(pr(w,r-1)w(r))=sl(pr(w,r))=a^s\gamma$ and $sl(pr(w,r-1))$ does not begin with $ a^s$. Then, by Lemma \ref{Lema 5}, $sl(pr(w,r-1))=a^{s-1}w'$ and we set $l_s=r$. Now, we can assume that the result is true for $pr(w,r-1)$ by induction hypothesis and we conclude the proof.
\end{proof}

\begin{obs}\label{Obs}
	Notice that with an analogous strategy we can prove the following statement. Let $w=b^{\pm 1}w'$ be geodesic such that $sl(w)=a^s\tilde{w}$, $s<0$.
	There exist prefixes $pr(w,l_1),...,pr(w,l_s)$, $l_1<l_2<...<l_s$, such that $sl(pr(w,l_i))$ begins with $a^i$, but $sl(pr(w,l_i-1))$ begins with $a^{i+1}$ but not by $a^i$,  $i=-1,...,s$.
\end{obs}

\begin{lem}\label{Lema 7}
	In a large even Artin group, let $s\geq t\geq 1$ (resp. $s\leq t\leq -1$) and \begin{equation}\label{eq1}
		w=b^tw', \hat{w}=a^s\hat{w}'    \text{ geodesic words such that } w=_G \hat{w}.
	\end{equation}
	
	Let $2m$ be the label between $a$ and $b$. The minimal length of such a word $w$ is $|w|=|s|+|t|(2m-1)$. Moreover, the only word $w$ of this length satisfying (\ref{eq1}) is: $$w=b^t[_{2m-1}(a,b)]^ta^{s-t}, \hat{w}= a^s[_{2m-1}(b,a)]^t$$ for $s,t\geq 1$ or  $$w=b^t[_{2m-1}(a^{-1},b^{-1})]^{|t|}a^{-|s-t|}, \hat{w}= a^s[_{2m-1}(b^{-1},a^{-1})]^{|t|}$$ for $s,t\leq -1$.
\end{lem}

\begin{proof}
	
	We will only prove the case $s,t\geq 1$, the case $s,t\leq -1$ follows by symmetry (it is enough to consider $a^{-1}, b^{-1}$ as generators instead of $a,b$).
	
	Without loss of generality, we fix a lexicographic order whose first letters are $a<b<a^{-1}<b^{-1}$. Then, as we may assume $\hat{w}'\in W$ we have $sl(w)=\hat{w}=a^s \hat{w}'$.
	
	By Corollary \ref{Corol 6} there exists a prefix $pr(w,l_1)$ such that $sl(pr(w,l_1))=a\alpha'$ is a geodesic word such that $\alpha'$ doesn't have a geodesic representative beginning with $a$ and $sl(pr(w,l_1-1))$ doesn't begin with $a$. Therefore, since the order is $a<b<a^{-1}<b^{-1}$ and $w$ begins with $b$, $sl(pr(w,l_1-1))$ must also begin with $b$.

	We claim that $sl(pr(w,l_1-1))$ must begin with $b^t$. Suppose that this does not happen, then there exists $l_2<l_1$ such that $sl(w,l_2-1)$ begins with $b^t$ but $sl(w,l_2)$ begins with $b^r$ with $r<t$, where $b^r$ is the biggest prefix of $sl(pr(w,l_2))$ that is a power of $b$. Observe that by construction $pr(w,l_2-1)$ begins with $b^t$. The fact that this element has a geodesic representative beginning with $b^t$ and that our order is $a<b<a^{-1}$ implies that $b^r a$ must be a prefix of $sl(pr(w,l_2))$. Then there is a leftward lex reducing sequence $$sl(pr(w,l_2-1))w(l_2) \rightarrow sl(pr(w,l_2))=b^r a \gamma.$$
	
	So if $u_n \rightarrow\tau(u_n)$ is the last $\tau$-move in the leftward lex reducing sequence, we have that $f(\tau(u_n))=a$. So, we know by Proposition \ref{cons} that $u_n$ must be a positive critical word and it must have one of the following forms:
	$$u_n=\begin{cases}
	\text{either } &b^{t-r} \xi (t,z)_{2m}\hspace{0.3 cm} \text{if } t-r\geq 1, \{z,t\}=\{a,b\} \\
	
	\text{or } &(b,a)_{2m}\hspace{0.1cm} \xi  \hspace{0.8 cm} \text{if } t-r=1
	
	\end{cases}$$
	
	In both cases $u=b^ru_n$ is a prefix of a word representative of $pr(w,l_2)$. If we are in the first case, then $u=b^r u_n=b^t \xi (t,z)_{2m}$ is also a critical word and applying $\tau$, we obtain a geodesic representative of $pr(w,l_2)$ that begins with $a$, giving a contradiction.
	
	If we are in the second case, $u=b^r u_n=b^{t-1} (b,a)_{2m}\xi$, applying $\tau$ we get $\tau(u_n)=\xi (t,z)_{2m}$ with $\{z,t\}=\{a,b\}$. Therefore, $u=_G  b^{t-1} \xi(t,z)_{2m}w'$. But notice that $b^{t-1} \xi(t,z)_{2m}$ is a critical word and if we apply $\tau$ we get a geodesic representative of $pr(w,l_2)$ that begins with $a$ which is again impossible. So our claim that $sl(pr(w,l_1-1))$ must begin with $b^t$ is proved.
	
	Now, by Proposition \ref{Simpl} we know that there is a leftward lex reducing sequence that transforms $sl(pr(w,l_1-1))w(l_1)$ into $sl(pr(w,l_1))$. As $sl(pr(w,l_1-1))w(l_1)$ begins with $b^t$ and $sl(pr(w,l_1))$ begins with $a$, the sequence must involve the first letter of these words and if we let $r$ be the length of that sequence, $\tau(u_r)=$ $_{2m-1}(a,b) b^t\xi$. Therefore,
	\begin{equation}\label{eq2}
		w=_{G} [_{2m-1}(a,b)] b^t\tilde{w}=\hspace{0.1cm}_{2m}(a,b) b^{t-1} \tilde{w}.
	\end{equation}
	Besides, this (together with the lexicographic order) implies that $_{2m-1}(a,b) b^t=$ $_{2m}(a,b) b^{t-1}$ is a prefix of $sl(pr(w,l_1))$.
	
	Now we will argue by induction on $t$. If $t=1$, notice that $_{2m-1}(a,b) b^t$ has length $2m$ and since by Corollary \ref{Corol 6} we will need at least $s$ different prefixes, the length of a word satisfying this must be at least $2m+s-1$, that is exactly the length of $_{2m-1}(b,a)a^s=_{G} a^s\hspace{0.1cm} _{2m-1}(b,a)$. Obviously, there are no other words of this length satisfying the property.
	
	Let us suppose now that the result is true for $t<k$ and consider $t=k$. Since $t>1$, we have that also $s>1$. We already know that $w=_{G} $ $_{2m-1}(a,b) b^t\tilde{w}=\hspace{0.1cm}_{2m}(a,b) b^{t-1} \tilde{w}$.
	
	 Now we claim that $sl(b^{t-1}\tilde{w})$ begins with $a^{s-1}$. For suppose not. Then $sl(b^{t-1}\tilde{w})=a^{p}\tilde{w}'$ with $0\leq p<s-1$ where $\tilde{w}'$ has no geodesic representative that begins with $a$. Besides, we know that $$w=_{G} \text{   }_{2m-1}(a,b)b b^{t-1}\tilde{w}=_{G} a _{2m-2}(b,a)ba^p \tilde{w}'=_G a^{p+1} \underbrace{_{2m-1}(b,a)\tilde{w}'}_{\alpha}=_G a^s w'.$$ We define $\alpha$ to be $_{2m-1}(b,a)\tilde{w}'$, as indicated in the displayed equation. As $p+1<s$, $sl(\alpha)$ must begin with $a$, and so there must be a prefix $pr(\alpha, l'_1)$ such that $sl(pr(\alpha, l'_1))$ begins with $a$ but $sl(pr(\alpha, l'_1-1))$ begins with $b$ because our lexicographic order is $a<b<a^{-1}$. 
	
	Therefore, there exists a leftward lex reducing sequence 
	$$\underbrace{sl(pr(\alpha, l'_1-1))}_{b...}\alpha(l_1') \longrightarrow \underbrace{sl(pr(\alpha, l'_1))}_{a...}.$$
	Let us denote by $u_1,..., u_r$ the critical words of the sequence. As  $sl(pr(\alpha, l'_1-1))\alpha(l_1')$ begins with $b$ and $sl(pr(\alpha, l'_1))$ begins with $a$, the end of the sequence must involve the first letter of the words. Since $\tilde{w}'$ does not have any geodesic representative beginning with $a$, the last critical word of the sequence, $u_r$, must be of the form $\xi (b,a)_{2m}$ with $\xi= $ $_{2m-1} (b,a) \xi'$  ($\xi'$ possibly empty) and it must be a positive critical word. Notice that if $\xi'\neq\epsilon$, $f[\xi']=b$ because by the form of the critical word it must begin with $b$ or $a$. The latter case is impossible since $\tilde{w}'$ doesn't have a geodesic representative beginning by $a$. 
	So we are left with two cases 
	\begin{itemize}
		\item either $\xi'=\epsilon$ and $f[u_r]=b$ \\
		\item or  $\xi'\neq \epsilon$.
	\end{itemize}
	
	Then in both cases  $u_r'=\xi' (b,a)_{2m}$ is also a critical word, and $\tau(u_r')$ begins with $a$. Thus the sequence $u_1,...,u_{r-1}, u_r'$ is a reducing sequence for $\tilde{w}'$. After applying $\tau$ we get a geodesic representative beginning with $a$, yielding a contradiction. That is, $sl(b^{t-1}\tilde{w})$ has the form $a^{s-1}\tilde{w}'$, for some word $\tilde{w}'$. 
	
	Therefore, we have the words $$w=_{G}   [_{2m}(a,b)] b^{t-1} \tilde{w}=_{G} [_{2m}(a,b)] a^{s-1} \tilde{w}'=_{G} a^{s} [_{2m-1}(b,a)]\tilde{w}'.$$
	
	So, the minimal length is $2m+\rm ML\it(s-1,t-1)$ where $\rm ML\it(s-1,t-1)$ is  the minimal length of the case  with parameters $t-1$, $s-1$, and by induction hypothesis we know that $\rm ML\it(s-1,t-1)=s-1+(t-1)(2m-1)$ and the only words satisfying this are $b^{t-1}[_{2m-1}(a,b)]^{t-1}a^{s-t}=_G a^{s-1}[_{2m-1}(b,a)]^{t-1}$.
	
	Thus, for parameters $t$, $s$, the minimal length is $2m+\rm ML\it(s-1,t-1)=2m+s-1+(t-1)(2m-1)=s+t(2m-1)$ and the only words of this length satisfying the conditions on the statement are $b^{t}[_{2m-1}(a,b)]^{t}a^{s-t}=_{G} a^{s} [_{2m-1}(b,a)]^{t}$ as we wanted.
	\end{proof}

\begin{lem}\label{Lema 8}
	In a large even Artin group we cannot have two geodesic words $a^s w_1 =_{G} b^{-t} w_2$, $s,t\geq 2$.
\end{lem}

\begin{proof}
	We may assume $s=t=2$. We argue by contradiction. Let $T$ be the set of elements of $G$ that have geodesic representatives $w'=a^2w_1 =_{G} b^{-2}w_2=w$ and assume $T\neq\emptyset$. Let $g\in T$ be an element of minimal geodesic length represented by $w=_G w'$ as before.
	
	Without loss of generality we consider the lexicographic order $a<b<b^{-1}<a^{-1}<...$. By the choice of the order we can suppose that $w'=sl(w)$. Besides, also by the choice of the order and the fact that $w$ is geodesic together with Lemma \ref{Lema 1} we deduce that if $sl(pr(w,i))$, $1\leq i \leq |g|$, doesn't begin with $b^{-1}$ it must begin with $a$.
	
	Let $n=|g|$, we have $sl(pr(w,n))=w'$. Since $g$ is minimal we know that $sl(pr(w,n))$ is the first prefix in the series beginning with $a^2$ (in other case there would be a prefix $\alpha$ of $w$ shorter than $w$ and such that $\alpha\in T$). Let $sl(pr(w,j))$ be the first prefix which doesn't begin with $b^{-2}$. Then, $sl(pr(w,j-1))$ begins with $b^{-2}$ and $sl(pr(w,j))$ may begin with $a$ or by $b^{-1}a$ (by the choice of the lexicographic order). Let us see that in both cases, we have that $sl(pr(w,n-1))$ must begin with $ab$.
	
	By Theorem \ref{Simpl} there is a leftward lex reducing sequence that transforms $sl(pr(w,j-1))w(j)$ into $sl(pr(w,j))$.
	
	\begin{enumerate}
		
		\item Let us suppose that $sl(pr(w,j))$ begins with $a$, then we know that the last critical word of the reducing sequence must involve the first letter. Let $r$ be the length of the reducing sequence, then $sl(pr(w,j))$ begins with $\tau(u_r)$. Then $u_r$ is a critical word in $a,b$ beginning with $b^{-2}$ and such that $\tau(u_r)$ begins with $a$. Let $2m$ be the label between $a$ and $b$, therefore: $u_r= b^{-1} b^{-1}\xi' (t,z)_{2m-1}$ where $\{t,z\}=\{a,b\}$. Thus, $\tau(u_r)= \hspace{0.1cm}_{2m-1}(a,b) b^{-1}\xi' t^{-1}$ where $t\in\{a,b\}$ so $sl(pr(w,j))$ begins with $ab$.
		
		Therefore, since by minimality $sl(pr(w,n-1))$ cannot begin with $a^2$, it begins with $ab$ (since $sl(pr(w,j))$ begins with $ab$ and $b$ is the second letter in the lexicographic order).

		\item Assume now that $sl(pr(w,j))$ begins with $b^{-1}a$. Recall that $sl(pr(w,j-1))$ begins with $b^{-2}$. Then we know that there exists a  leftward lex reducing sequence transforming $sl(pr(w,j-1))w(j)$ into $sl(pr(w,j))$, and that the last critical word of the reducing sequence must involve the second letter of the word and not the first one (i.e. every letter of the word except the first one belongs to the reducing sequence). Let $r$ be the length of the reducing sequence, then $sl(w,j)$ begins with $b^{-1}\tau(u_r)$. Then $u_r$ is a critical word in $a,b$ beginning with $b^{-1}$ and such that $\tau(u_r)$ begins with $a$, therefore: $u_r= \hspace{0.1cm}_{n}(b^{-1},a^{-1}) \xi (t,z)_{p}$ where $\{t,z\}=\{a,b\}$, $p+n=2m$. Thus, $\tau(u_r)= \hspace{0.1cm} _{p}(a,b) \xi (t^{-1},z^{-1})_n$.
		
		By Corollary \ref{Corol 6} there exists $k$ with $j<k<n$ such that $sl(pr(w,k))$ begins with $a$ and $sl(pr(w,k-1))$ begins with $b^{-1}a$ (by the chosen lexicographic order and Lemma \ref{Lema 1}). We know by Theorem \ref{Simpl} that there exists a leftward lex reducing sequence that transforms $sl(pr(w,k-1))w(k)$ into $sl(pr(w,k))$, let $d$ be the length of this sequence. Then, since the first letter is changed, $sl(pr(w,k))$ must begin with $\tau(u_d)$, besides recall that $sl(pr(w,k-1))$ begins with $b^{-1}a$. Therefore, $u_d$ has one of the following forms:
			$$u_{d}=\begin{cases}
		\text{either } &b^{-1} (a\xi') \hspace{0.1cm}_{2m-1}(t,z), \hspace{0.5 cm} \{t,z\}=\{a,b\} \\
		
		\text{or } &b^{-1}\hspace{0.1cm} _{2m-1}(a,b) \hspace{0.3 cm} 
		
		\end{cases}$$
		
		So, then:		
		$$\tau(u_{d})=\begin{cases}
		\text{either } &_{2m-1}(a,b) (a\xi') t^{-1},  \hspace{0.5 cm}  t\in\{a,b\}\\
		
		\text{or } &_{2m-1}(a,b) b^{-1} \hspace{0.3 cm} 
		
		\end{cases}$$
		
		Therefore, $sl(pr(w,k))$ begins with $ab$. Moreover $sl(pr(w,n))$ begins with $a^2$ and $sl(pr(w,n-1))$ doesn't by minimality. Therefore, $sl(pr(w,n-1))$ must also begin with $ab$ (since $b$ is the second letter in the lexicographic order).
		
	\end{enumerate}
	
	So we have that $sl(pr(w,n))$ begins with $a^2$ and $sl(pr(w,n-1))$ begins with $ab$. By Theorem \ref{Simpl} we know that there is a leftward lex reducing sequence that transforms $sl(pr(w,n-1))w(n)$ into $sl(pr(w,n))$, and if we call $r'$ to the length of that sequence $sl(pr(w,n))=a\tau(u_{r'})$. Therefore, by Proposition \ref{cons} $u_{r'}$ is a positive critical word beginning with $b$. Then
	$$\tau(u_{r'})=\begin{cases}
	\text{either } &a \xi'\hspace{0.1cm} _{2m}(t,z),  \hspace{0.5 cm}  \{z,t\}=\{a,b\}\\
	
	\text{or } &_{2m}(a,b) \xi \hspace{0.3 cm} 
	
	\end{cases}$$
	
	Thus:
	$$u_{r'}=\begin{cases}
	\text{either } &_{2m}(b,a) a\xi' \hspace{0.3 cm}  \\
	
	\text{or } &\xi \hspace{0.1cm}_{2m}(t,z), \hspace{0.5 cm} \{z,t\}=\{a,b\}, 
	
	\end{cases}$$
	
	In both cases, $au_{r'}$ is also a positive critical word beginning with $a$, therefore applying $\tau$ to this critical word instead of applying it to $u_{r'}$ in the last step, we would obtain a geodesic representative of $w$ beginning with $b$ (by Proposition \ref{cons}). But $w$ begins with $b^{-1}$ and this is impossible by Lemma \ref{Lema 1}.
\end{proof}

\begin{lem}\label{Lema 9}
	In a large even Artin group, let $w=b^tw'$ (resp. $w=a^{-t}\hat{w}'$), $\hat{w}=a^{-1}\hat{w}'$ (resp. $\hat{w}=bw'$), $ t\geq 1$ be geodesic words such that $w=_G \hat{w}$. Let $2m$ be the label between $a$ and $b$. The minimal length of such a word $w$ is $|w|=t+(2m-1)$. Moreover, the only words $w$ of this length satisfying that are: $w=b^t\hspace{0.1cm}_{2m-1}(a^{-1},b^{-1})=_G$ $ _{2m-1}(a^{-1},b^{-1}) b^t$ (resp. $w=a^{-t}\hspace{0.1cm}_{2m-1}(b,a)=_G$ $_{2m-1}(b,a)a^{-t}$).
\end{lem}

\begin{proof}
	Obviously, a word $w$ like that must begin with $b^t$ (resp. $a^{-t}$ ) and since it can be transformed, must contain a critical word, so it must have at least length $t+2m-1$.
	
	The only way to get a word like that of this length is considering the shortest possible critical word, i.e. $w=b^t\hspace{0.1cm}_{2m-1}(a^{-1},b^{-1})$ (resp. $a^{-t}\hspace{0.1cm}_{2m-1}(b,a)$).
\end{proof}

Finally, we will need the following result which is equivalent to Proposition 4.5 (3) in \cite{Holt}.

Given a set $A$, we will use $\#A$ to denote the cardinal of the set.

\begin{lem}\label{Initials}
	Let us consider the large even Artin group $A_{\Gamma}$. Given an element $g\in A_{\Gamma}$  $$\#\{v\in V(\Gamma)^{\pm}\mid \text{exists } w \text{ geodesic representative of } g \text{ such that } f[w]=v\}\leq 2.$$
	
\end{lem}

\section{Technical key result}\label{S3}

In this section we will find a presentation for the kernel of the map $\psi: A_{\Gamma} \rightarrow A_{\Gamma\setminus\{r\}}$ induced by $r\mapsto 1$ and $v\mapsto v$ if $v\neq r$. Notice that this is well defined because our group is an even Artin group.

\begin{notat}
Let $\Gamma$ be a simple labeled graph with even labels. Let $r\in V(\Gamma)$. We will consider the Artin groups $G=A_{\Gamma}$ and $G_1=A_{\Gamma\setminus\{r\}}$. We will denote by $b_{1},...,b_{n}$ the vertices connected to $r$ with label $2k_j$, $k_j>1$, $1\leq j\leq n$ and by  $c_1,...,c_{k}$ the vertices non-connected to $r$. 
\end{notat}

\begin{mydef}\label{parametros}
For $i=1,...,n$ we define the following integer numbers:
\begin{align}
&p_i^+=\left \lfloor {\frac{k_i}{2}} \right \rfloor +1 \nonumber\\
&n_i^-=-\left(\left \lfloor {\frac{k_i-1}{2}} \right \rfloor+1\right)\nonumber\\
&p_i^-=\left \lfloor {\frac{k_i}{2}} \right \rfloor+1-k_i=p_i^+-k_i\nonumber\\
&n_i^+=k_i-\left(\left \lfloor {\frac{k_i-1}{2}} \right \rfloor+1\right)=k_i+n_i^-\nonumber
\end{align}
where $\left\lfloor{x}\right\rfloor$ denotes the integer part of $x$.
\end{mydef}

\begin{obs}\label{obs1}
	
	The following properties are satisfied for large even Artin groups (every $k_i\geq 2$):
	\begin{enumerate}
	\item $p_i^+\geq 2$, $n_i^+> 0$, $n_i^-<0, p_i^-\leq 0$.
	\item $p_i^+>|p_i^-|$
	\item $|n_i^-|\geq n_i^+$ and the equality holds if and only if $k_i$ is even.
	\item $p_i^+\geq |n_i^-|$ and the equality holds if and only if $k_i$ is odd.
	\end{enumerate}
\end{obs}

\begin{lem}\label{numeritos}
	We have:
	\begin{enumerate}
		\item $p_i^--1=n_i^-$
		\item $p_i^+-1=n_i^+$
	\end{enumerate}
\end{lem}

\begin{proof}
	\hspace{0.5 cm}
	\begin{enumerate}
		\item $p_i^--1=\left(\left\lfloor{\frac{k_i}{2}}\right\rfloor +1-k_i\right)-1=\left\lfloor{\frac{k_i}{2}}\right\rfloor-k_i=-\left(\left\lfloor{\frac{k_i-1}{2}}\right\rfloor+1\right)=n_i^-$.
		\item
	Notice that $p_i^+-p_i^-=n_i^+-n_i^-=k_i$. Then, $p_i^+-n_i^+=p_i^--n_i^-$, and by the previous point, we have that $p_i^--n_i^-=1$. Therefore, $p_i^+-n_i^+=1$ as we wanted to prove.
	\end{enumerate}
\end{proof}

\begin{mydef}\label{omegas}
	We will consider the following sets of elements of $G$:
	\begin{align}
	& \Omega_i^+=\left\{g\in G \mid  g \text{ has a geodesic representative beginning with } b_i^{p_i^+}\right\}\nonumber \\
	& \Omega_i^-=\left\{g\in G \mid g \text{ has a geodesic representative beginning with } b_i^{n_i^-}\right\}\nonumber 
	\end{align}
	for $i=1,...,n$.
\end{mydef}

\begin{obs}\label{obs2}
	
	We have $\Omega_i^+\cap \Omega_i^-=\emptyset$ by Lemma \ref{Lema 1}.

\end{obs}

Consider an Artin relation of even type of the form $(rb_i)^{k_i}=(b_ir)^{k_i}$. We can rewrite it as follows:
$$r^{b_i^{k_i}}=r^{b_i^{k_i-1}}...r^{b_i}r\left(r^{b_i}\right)^{-1}...\left(r^{b_i^{k_i-1}}\right)^{-1}$$
$$r^{-b_i}=r^{-1}\left(r^{b_i}\right)^{-1}...\left(r^{b_i^{k_i-2}}\right)^{-1}r^{b_i^{k_1-1}}r^{b_i^{k_i-2}}...r^{b_i}r$$

From here we can easily see that these expressions are equivalent to:
$$r^{b_i^{p_i^+}}=r^{b_i^{p_i^+-1}}...r^{b_i^{p_i^-}}...\left(r^{b_i^{p_i^+-1}}\right)^{-1}$$
$$r^{b_i^{n_i^-}}=\left(r^{b_i^{n_i^-+1}}\right)^{-1}...r^{b_i^{n_i^+}}...r^{b_i^{n_i^-+1}}$$

Taking these relations as inspiration, we define the following sets of relations:

\begin{mydef}\label{Relaciones}
	\begin{align}
	&\hat{R}^{+}=\left\{r^{b_{i}^{p_i^+}g}=r^{b_{i}^{p_i^+-1}g}...r^{b_i^{p_i^-}g}...\left(r^{b_{i}^{p{i}^+-1}g}\right)^{-1};  b_i^{p_i^+}g\in\Omega_i^{+} \text{ for some } 1\leq i \leq n\right\} \nonumber\\
	&\hat{R}^{-}=\left\{r^{b_{i}^{n_i^-}g}=\left(r^{b_i^{n_i^-+1}g}\right)^{-1}...r^{b_i^{n_i^+}g}...r^{b_i^{n_i^-+1}g}; b_i^{n_i^-}g\in\Omega_i^{-} \text{ for some } 1\leq i \leq n\right\} \nonumber\\
	& \hat{R}=\hat{R}^{+}\cup \hat{R}^{-} \nonumber                     
	\end{align}
\end{mydef}

\begin{obs}\label{obs5}
	
	Let $h=b_i^{p_i^+}\in \Omega_{i}^+$. The relation of $\hat{R}^+$ associated to $h$ is  $r^{b_{i}^{p_i^+}g}=r^{b_{i}^{p_i^+-1}g}...r^{b_i^{p_i^-}g}...\left(r^{b_{i}^{p{i}-1}g}\right)^{-1}$ and will be denoted by $R(h)$.
	
	Analogously, let $h=b_{i}^{n_i^-}g\in \Omega_i^-$. The relation of $\hat{R}^-$ associated to $h$ is $r^{b_{i}^{n_i^-}g}=\left(r^{b_i^{n_i^-+1}g}\right)^{-1}...r^{b_i^{n_i^+}g}...r^{b_i^{n_i^-+1}g}$ will and will be denoted by $R(h)$.
\end{obs}

Now, let us consider the map $$\psi: A_{\Gamma} \rightarrow A_{\Gamma_{\setminus \{r\}}}$$ induced by $r\mapsto 1$ and $v\mapsto v$ if $v\neq r$. We want to obtain a presentation for the kernel of this map. To do so, we will observe that $A_{\Gamma}\simeq A_{\Gamma_{\setminus\{r\}}} \ltimes \ker(\psi)$ and use a standard argument that can be found in Appendix A of \cite{Conchita}, to obtain a presentation for the semidirect product.

Let $K=\langle Y\mid C\rangle$ and $G=\langle Z \mid T \rangle$ be groups. Let $G$ act on $Y$ by permutations. Notice that $C\leq F(Y)$, the free group generated by $Y$, and observe that $G$ also acts on $F(Y)$. We assume that this action preserves $C$. Let $Y_0$ be a set of representatives for the $G$-orbits in $Y$ and $C_0$ be a set of representatives for the $G$-orbits in $C$. We observe that $C_0\leq \langle t(a_0) \mid a_0\in Y_0, t\in G\rangle$ that is, we may express elements of $C_0$ as products of elements in the $G$-orbit of $Y_0$. We then set $\hat{C}_0\subset \langle t^{-1} a_0 t \mid a_0\in Y_0, t\in G \rangle$ to be the set of fixed expressions for the elements of $C_0$ where we have replaced the action of $G$ on $Y_0$ by the conjugation of elements. The set $\hat{C}_0$ is thus a set of formal expressions which will be used later to express relations in groups.

\begin{lem}\label{subgroup}\cite{Conchita}
	With the notation above, we have $$G\ltimes K= \langle Y_0, Z\mid \hat{C}_0, T, [Stab_G(y),y], y\in Y_0\rangle,$$
	where the semidirect product is given by the action of $G$ on $K$.
\end{lem}

\begin{lem}\label{Commute}
	Let $G$ be a large even Artin group, $r$ a generator and $G_1=A_{\Gamma\setminus\{r\}}$ as before. If $g\in G_1$ satisfies $g^{-1}rg=r$, then $g=\epsilon$.
\end{lem}

\begin{proof}
	Let $w$ be a shortlex representative for $g$. As $r=_G w^{-1}rw$, the word $w^{-1}rw$ is obviously not shortlex, so it is clear that there exists a prefix $\alpha$ of $w^{-1}rw$ such that $\alpha$ contains $r$ and admits a rightward length reducing sequence. But that means that at some moment of the sequence we will have a critical word $u_i$ in two letters, one of them being $r$. Since $G$ is large even, in every critical word the name of each of the letters appears at least twice, so one of the occurrences of $r$ must occur in $w$ or $w^{-1}$. But since $g\in G_1$, it is impossible that either $r$ or $r^{-1}$ appear in $w$ or $w^{-1}$. So it is impossible to have such a critical word $u_i$ and hence $g=\epsilon$.
\end{proof}

\begin{lem}\label{action}
$\ker(\psi)$ is isomorphic to:
	 $$K:= \left\langle r^{g}; g\in A_{\Gamma_1} \mid \hat{R}\right\rangle,$$ 
	 where $\hat{R}$ is the set of relations defined in Definition \ref{Relaciones}.
\end{lem}

\begin{proof}
	Note that $\ker(\psi)$ is the normal subgroup of $A_{\Gamma}$ generated by $r$. 	We define an action of $G_1=A_{\Gamma\setminus\{r\}}=\langle S_1\mid C\rangle$ (where $S_1$ is the set of generators of $G_1$ and $C$ is the set of Artin relations of $G_1$) on the abstract group $K$ via
	$$h^{-1}(r^{g})h=r^{gh}, \hspace{0.3 cm} h\in G_1.$$
	
	Let us see that this action preserves the relators of $K$. To prove that, we may suppose that $h$ is a generator,  i.e. that $h\in S_1$. 	Consider an $\hat{R}^+$ relation $$r^{b_{i}^{p_i^+}g}=r^{b_{i}^{p_i^+-1}g}...r^{b_i^{p_i^-}g}...\left(r^{b_{i}^{p{i}-1}g}\right)^{-1},$$ where $g$ has a representative $w$ with $b_i^{p_i^+}w$ geodesic. If $h$ acts on both sides of this relation we obtain:
$$r^{b_{i}^{p_i^+}gh} \text{ and } r^{b_{i}^{p_i^+-1}gh}...r^{b_i^{p_i^-}gh}...\left(r^{b_{i}^{p{i}-1}gh}\right)^{-1},$$
we want to see that these two elements of $K$ are equal.
	
 Assume first that the element $gh$ has some geodesic representative $u$ with $b_i^{p_i^+}u$ also geodesic. Then there is a $\hat{R}^+$ relation of the form $$r^{b_{i}^{p_i^+}gh}=r^{b_{i}^{p_i^+-1}gh}...r^{b_i^{p_i^-}gh}...\left(r^{b_{i}^{p{i}-1}gh}\right)^{-1},$$ which is precisely the image under $h$ of the previous relation.
	
	Now, we are left with the case when $g$ has a representative $w$ with $b_i^{p_i^+}w$ geodesic but $b_i^{p_i^+}wh$ has no geodesic representative beginning with $b_i^{p_i^+}$. Then, by Lemma \ref{Lema 3} $gh$ has a geodesic representative that begins with $b_i^{-1}$, say $b_i^{-1}\alpha$ . Thus, $b_i^{p_i^+}gh$ has a geodesic representative that begins with $b_i^{p_i^+-1}$. But, by Lemma \ref{numeritos} we have that $p_i^+-1=n_i^+$. In the same way, we can see that $b_i^{p_i^-}gh$ has a geodesic representative that begins with $b_i^{p_i^--1}$ and that, again by Lemma \ref{numeritos}, $p_i^--1=n_i^-$. So, it begins with $b_i^{n_i^-}$ and rearranging, we obtain:
	$$r^{b_i^{p_i^-}gh}=r^{b_i^{n_i^-}\alpha}=\left(r^{b_i^{n_i^-+1}\alpha}\right)^{-1}...r^{b_i^{n_i^+}\alpha}...r^{b_i^{n_i^-+1}\alpha}$$
	 which is a relation of $\hat{R}^{-}$.
		
	Following the same strategy, we can prove that the same happens with the $\hat{R}^{-}$ relations. 
	
	Now, applying Lemma \ref{subgroup} to our case, we obtain that $$G_1\ltimes K = \langle Y_0, Z\mid \hat{C}_0, T, [Stab_{G_1}(y),y], y\in Y_0\rangle,$$ where $Y_0=r$, $Z=V(G_1)$, $\hat{C}_0=\hat{C}^+\cup \hat{C}^-$, $T=C$ with
\begin{align*}
&\hat{C}^{+}=\{(b_{i}^{p_i^+}g)^{-1}r (b_{i}^{p_i^+}g)=(b_{i}^{p_i^+-1}g)^{-1}r (b_{i}^{p_i^+-1}g)...(b_i^{p_i^-}g)^{-1} r (b_i^{p_i^-}g)...((b_{i}^{p{i}-1}g)^{-1}r (b_{i}^{p{i}-1}g))^{-1}; \nonumber \\ &\hspace{1.2 cm} b_i^{p_i^+}g\in\Omega_i^{+} \text{ for some } 1\leq i \leq n\} \nonumber\\
&\hat{C}^-=\{(b_{i}^{n_i^-}g)^{-1}r (b_{i}^{n_i^-}g)=((b_i^{n_i^-+1}g)^{-1}r (b_i^{n_i^-+1}g))^{-1}...(b_i^{n_i^+}g)^{-1}r (b_i^{n_i^+}g)...(b_i^{n_i^-+1}g)^{-1}r (b_i^{n_i^-+1}g); \nonumber \\&\hspace{1.2 cm} b_i^{n_i^-}g\in\Omega_i^{-} \text{ for some } 1\leq i \leq n\} \nonumber\\               
\end{align*}

Therefore, $$G_1\ltimes K =\langle r, V(G_1)\mid \hat{C}_0, C, [Stab_{G_1}(r),r] \rangle.$$

Now, if $g\in Stab_{G_1}(r)$, we have that $g^{-1}rg=r^{sl(g)}=r$ in $K$ because of the form of the action. Hence $sl(g)=\epsilon$ by Lemma \ref{Commute} and thus we don't have any relation of this type. Thus:
	$$G\ltimes K= \langle r, V(G_1)\mid \hat{C}_0, C \rangle.$$

	We define:
	$$C'_0=\{(b^{p_i^+})^{-1}r(b^{p_i^+})=(b^{p_i^+-1})^{-1}r(b^{p_i^+-1})...r...((b^{p_i^+-1})^{-1}r(b^{p_i^+-1}))^{-1}, i=1,...,n\}$$

	So in fact, the relations in $\hat{C}_0\setminus C'_{0}$ are obtained from the ones of $C'_{0}$ by conjugation. So we can eliminate them from the presentation using Tietze transformations. Thus, we have:
		$$G_1\ltimes K=\langle V(G_{1}),r\mid C'_0, C\rangle \simeq A_{\Gamma}$$ and the isomorphism maps $K$ onto $\ker(\psi)$.	Therefore, $K  :=\langle r^{sl(g)}; g\in G_1 \mid \hat{R}\rangle= \langle\langle r\rangle\rangle= \ker(\psi)$. 
\end{proof}

\section{Poly-freeness for large even Artin groups}\label{sec}

In this section we are going to prove our main result, i.e. that every large even Artin group is polyfree.

Let $\Gamma$ be a labeled graph with even labels $\geq 4$ and $A_{\Gamma}$ the associated Artin group. Let $r$ be a vertex of the graph and $G_1=A_{\Gamma\setminus \{r\}}$

By Lemma \ref{action} we know that:
$$K:=\langle\langle r\rangle \rangle=\langle r^{g}; g\in G_1 \mid \hat{R}\rangle $$

with 
\begin{align}
&\hat{R}^{+}=\left\{r^{b_{i}^{p_i^+}g}=r^{b_{i}^{p_i^+-1}g}...r^{b_i^{p_i^-}g}...\left(r^{b_{i}^{p{i}-1}g}\right)^{-1};  b_i^{p_i^+}g\in\Omega_i^{+} \text{ for some } 1\leq i \leq n\right\} \nonumber\\
&\hat{R}^{-}=\left\{r^{b_{i}^{n_i^-}g}=\left(r^{b_i^{n_i^-+1}g}\right)^{-1}...r^{b_i^{n_i^+}g}...r^{b_i^{n_i^-+1}g}; b_i^{n_i^-}g\in\Omega_i^{-} \text{ for some } 1\leq i \leq n\right\} \nonumber\\
& \hat{R}=R^{+}\cup R^{-} \nonumber                     
\end{align}

where $b_1,...,b_n$ are the vertices connected to $r$, $2k_1,..., 2k_n$ the labels of the connecting edges and		\begin{align}
	& \Omega_i^+=\left\{h\in G \mid h \text{ has a geodesic representative beginning with } b_i^{p_i^+}\right\}\nonumber \\
	& \Omega_i^-=\left\{h\in G \mid h \text{ has a geodesic representative beginning with } b_i^{n_i^-}\right\}\nonumber 
	\end{align}
	for $i=1,2,...,n$.

Recall that in Definition \ref{parametros} we have defined the numbers $p_i^+, p_i^-, n_i^-, n_i^+$ and that by Remark \ref{obs2} we already know that $\Omega_i^+\cap \Omega_i^-=\emptyset$.

\begin{lem}\label{Max3}
	An element $g\in G$ can belong to at most two different sets $\Omega_i^{\pm}$, $i=1,...,n$.
\end{lem}

\begin{proof}
	It is easily deduced from  Lemma \ref{Initials}.
\end{proof}

\begin{lem}\label{obs2b}
	If $i\neq j$, $\Omega_i^+ \cap \Omega_j^- \neq\emptyset$ if and only if $n_j^-=-1$, i.e. $k_j=2$.
\end{lem}

\begin{proof}
	Recall that $p_i^+\geq 2$. By Lemma \ref{Lema 8}, given a geodesic word $w=a^s w'$, $s\geq 2$, there cannot be another geodesic representative of the same element beginning by $b^{-t}$, $t>1$. So, if $\Omega_i^+ \cap \Omega_j^- \neq\emptyset$, then $n_j^-=-1$. The fact that if $n_j^-=-1$ then $\Omega_i^+ \cap \Omega_j^- \neq\emptyset$ is clear by Lemma \ref{Lema 9}.
\end{proof}

Now, notice that by the form of the relations in $\hat{R}$, we see that each $r^{b_i^{p_i^+}g}$ (resp. $r^{b_i^{n_i^-}g}$) is conjugate in $K$ to $r^{b_i^{p_i^-}g}$ (resp. $r^{b_i^{n_i^+}g}$ ).

Let us define the following maps:
\begin{align}
\rho_i^{\varepsilon}: \Omega_i^{\varepsilon} &\longrightarrow G_1 & \nonumber \\
h &\mapsto b_i^{-\varepsilon k_i}h \nonumber
\end{align}
for $i=1,...,n$ and $\varepsilon=\pm$.

Notice that by Definition \ref{parametros}, $p_i^+-p_i^-=n_i^+-n_i^-=k_i$. Therefore, if $\varepsilon=+$,since $h\in\Omega_i^{+}$, $h$ has a geodesic representative of the form $b_i^{p_i^+}\bar{h}$ where $\bar{h}$ doesn't have a geodesic representative beginning with $b_i^{-1}$. Thus, $\rho_i^+(h)=b_i^{p_i^-}\bar{h}$. Similarly, if $\varepsilon=-$, then $h$ has a geodesic representative of the form $b_i^{n_i^-}\bar{h}$ where $\bar{h}$ doesn't have a geodesic representative beginning with $b_i$. Thus, $\rho_i^-(h)=b_i^{n_i^+}\bar{h}$.

\begin{mydef}
Let $$\Omega=\cup_{i=1}^n (\Omega_i^+\cup\Omega_i^-).$$ We also define $\Lambda=G_1 \setminus \Omega$.

\end{mydef}

\begin{mydef}
Let $\mathcal{P}$ be the set of subsets of $G_1$. We set:
\begin{align}
\rho: \mathcal{P} &\longrightarrow \mathcal{P} \nonumber \\
A &\mapsto \cup_{i=1}^n (\rho_i^+(A\cap \Omega_i^+))\cup (\rho_i^-(A\cap\Omega_i^-))\cup (A\cap \Lambda).
\end{align}

When we consider the image under $\rho$  of a one element subset $\{g\}\subset \mathcal{P}(W)$, we will write $\rho(g)$ instead of $\rho(\{g\})$.
\end{mydef}

\begin{lem}\label{obs3}
	Let $g\in \Omega$.
	
	\begin{enumerate}[(i)]
		\item If $g$ lies in only one of the sets $\Omega_{i}^{\pm}$, say $g\in \Omega_{i}^{\epsilon}$, then $$\rho(g)=\{\rho_i^{\epsilon}(g)\}.$$
		
		\item If $g$ lies in two of the sets, say $g\in \Omega_i^{\epsilon}\cap \Omega_j^{\delta}$, then $$\rho(g)=\{\rho_i^{\epsilon}(g), \rho_j^{\delta}(g)\}.$$

	\end{enumerate}
	Note that by Lemma \ref{Max3} there are no other possibilities.

\end{lem}

\begin{proof}
	The result is just an inmediate consequence of the definition of $\rho$.
\end{proof}

\begin{lem}\label{obs4}
	Let $g\in \Omega_i^+$, $g_1\in \Omega_i^-$. We have
	
	\begin{enumerate}[(i)]
		\item $|g|>|\rho_i^+(g)|$
		\item $|g_1|\geq |\rho_i^-(g_1)|$ with equality if and only if $w_1=b_i^{n_i^-}w_2$ is a geodesic representative of $g_1$, $k_i$ is an even number and $w_2$ doesn't have a geodesic representative beginning with $b_i^{-1}$ (and therefore, it cannot begin with $b_i^{\pm 1}$).
	\end{enumerate}
\end{lem}

\begin{proof}
	The result follows from Remark \ref{obs1}.
\end{proof}

\begin{lem}\label{Lema 10}
	Let $g\in \cup_{q=1}^{n} \Omega_q^-$ such that $|g|=|\rho_i^{-}(g)|$. Then  $\rho_i^{-}(g)$ doesn't belong to the intersection of two of the sets $\Omega_j^{\pm}$, $j=1,2,...,n$.
\end{lem}

\begin{proof}
	The hypothesis and Lemma \ref{obs4} imply that $g$ has a geodesic representative of the form $w=b_i^{n_i^-}w'$ and that $k_i$ is an even number.
	
	Therefore, the element $\rho_i^-(g)$ has a geodesic representative of the form $b_i^{k_i+n_i^-}w'$ (because of the hypothesis on its length), where no geodesic representative of $w'$ begins with $b_i$.
	
	Assume that there exists $b_i^{p_i^+}u$ geodesic such that $b_i^{k_i+n_i^-}w'=_G b_i^{p_i^+}u$. This is impossible because $k_i+n_i^-<p_i^+$ by Lemma \ref{numeritos} and $w'$ doesn't have any geodesic representative beginning with $b_i$. So $\rho_i^{-}(g)\not\in \Omega_i^{+}$.
	
	Analogously, we see that $\rho_i^{-}(g)$ cannot be in $\Omega_i^{-}$ (because $|n_i^+|>0$ and $w'$ doesn't have any geodesic representative beginning with $b_i^{-1}$) and by Lemma \ref{Initials}  there is only other possible letter such that the word can begin with it so it can belong at most to one $\Omega_j^{\pm}$.
\end{proof}

\begin{lem}\label{LexCond}
	Assume that we choose a lexicographic order such that for every $j$, $v_j<v_j^{-1}$, then for every $g$ in the suitable set $sl(g)>_{slex} sl(\rho_i^{\pm}(g))$.
\end{lem}

\begin{proof}
	By Lemma \ref{obs4} we only have to consider the case when $g$ has a geodesic representative $w=b^{n_i^-}w'$, $k_i$ is an even number and $w'$ doesn't have any geodesic representative beginning with $b_i^{\pm 1}$.
	
	In this case, by definition of $\rho_{i}^{-}$ we know that $\rho_i^-(g)=b^{n_i^+}g'$ where $g'$ is the element represented by the word $w'$ and $|n_i^-|=|n_i^+|$. Notice that $sl(b_i^{n_i^-}w')>_{slex} sl(b_i^{n_i^+}w')$ if and only if $sl(b_i^{n_i^-}w'w'^{-1})=b_i^{n_i^-}>_{slex} sl(b_i^{n_i^+}w'w'^{-1})=b_i^{n_i*}$. But $n_i^-$ is a negative number and $n_i^+$ is positive and since in the lexicographic order we have $b<b^{-1}$, the result follows.
\end{proof}

\begin{lem}\label{Proposition 12}
	Given $g\in \Omega$ there exists $l\in\mathbb{Z}^+$ such that if $h\in \rho^l(g)=\underbrace{\rho(\rho...\rho}_l(g))$ and $h\in\Omega$ then $|h|<|g|$.
\end{lem}

\begin{proof}
	Let $g\in \Omega$, we know that $\#\rho(g)=1$ or $2$ (recall that we use $\#A$ to denote the cardinal of the set $A$.). We define $\beta=\{h\in\rho(w) \mid |h|=|g| \}$. If $\beta$ is empty, the result follows for $l=1$ by Lemma \ref{obs4}. So we can suppose $\beta$ is not empty. 
	
	By Lemma \ref{obs4} $g$ must have a geodesic representative $w= b_i^{n_i^-}w'$ for some $i$ such that $k_i$ is even and $w'$ doesn't have any geodesic representative beginning with $b_i^{\pm 1}$. Considering if necessary all the elements of $\beta$ instead of $g$ and using Lemma \ref{Lema 10} we may assume that $\#\rho^l(g)=1$ for all $l\geq 1$. It is enough to prove that  there exists $l\in\mathbb{Z}^+$ such that $|\rho^l(g)|<|g|$.

	We are going to argue by induction over $m$, the number of negative letters in $w$. Notice that since we are working on an even Artin group, the number of negative letters in a word is constant for any of its geodesic representatives (see Definition \ref{taumoves}).
	
	As $w= b_i^{n_i^-}w'$ for $k_i$ even, $|n_i^-|\leq m$. Let $k=\min \{|n_j^-|\mid k_j \text{ is even}\}$. If $m=k$, then $w'$ must be positive. But then, $sl(\rho(g))=sl(b_i^{n_i^+}w')$ is positive, and therefore if $\rho^2(g)\in \Omega$, $|\rho^2(g)|<|\rho(g)|=|g|$.
	
	In the general case, $sl(\rho(g))=sl(b_i^{n_i^+}g')$ has less negative letters than $w$, and we distinguish three cases: if $\rho(g)\not\in \Omega$ or $ |\rho(\rho(g))|<|\rho(g)|$  we are done. Otherwise, we may apply the induction hypothesis and we obtain that there exists an $l-1\in\mathbb{Z}^+$ such that either $\rho^{l-1}(\rho(g))\not\in\Omega$ or $|\rho^{l-1}(\rho(g))|<|\rho(g)|=|g|$. That is, either $\rho^l(g)\not\in\Omega$ or $|\rho^{l}(g)|<|g|$.
\end{proof}

\begin{corol}\label{Corol 13}
	Given $g\in \Omega$, there exists $l\in\mathbb{Z}^+$ such that $\rho^l(g)\subset \Lambda$.
\end{corol}

\begin{proof}
	 We will argue by induction over the geodesic length of $g$.
	 If $|w|=0$ the result is obvious.
	 
Let $\gamma=|g|$, and suppose that it is true for length less than $\gamma$. By Lemma \ref{Proposition 12} there exists $l_1\in\mathbb{Z}$ such that for every $g'\in\rho^{l_1}(g)$ either $g'\in \Lambda$ or $|g'|<|g|=\gamma$. Let $\beta=\{g'\in\rho^{l_1}(g)\mid g'\not\in \Lambda\}=\{g'_1,...,g'_\mu\}$, by induction hypothesis there is $l'_i\in\mathbb{Z}$ such that $\rho^{l'_i}(g'_i) \subset \Lambda$ for $i=1,...,\mu$. Let $l_2=\max\{l'_i\}$. Thus, for $l=l_1+l_2$ we have that $\rho^{l}(g)\subset \Lambda$.
\end{proof}

\begin{mydef}
	Given $g\in \Omega$ we define $\alpha(g)$ as the smallest positive integer such that $\rho^{\alpha(g)}(g)\subset \Lambda$.
	This way, we define the map:
		\begin{align}
	\delta: \Omega &\longrightarrow \mathcal{P}(\Lambda) \nonumber\\
	g &\mapsto \delta(g):= \rho^{\alpha(g)}(g) \nonumber
	\end{align}
\end{mydef}

\begin{mydef}
We denote $$H=\langle b_1^{k_1},b_2^{k_2},...,b_n^{k_n}\rangle.$$ 
\end{mydef} 
 
 \begin{obs}
 Notice that the elements in $\delta(g)$ lie in the intersection of the coset $Hg$ with $\Lambda$.
 \end{obs}
 
 \begin{lem}\label{subgroupH}
 	$H$ is freely generated by $b_1^{k_1},...,b_n^{k_n}$.
 \end{lem}
 
 \begin{proof}
 	
 	Observe first that any freely reduced word in the alphabet $\{b_1^{\pm k_1}, b_2^{\pm k_2},...,b_n^{\pm k_n}\}$ can also be seen as a freely reduced word in the alphabet $\{b_1^{\pm 1}, b_2^{\pm 1},..., b_n^{\pm 1}\}$. Now, assume that $w, w'$ are freely reduced words in $\{b_1^{\pm k_1}, b_2^{\pm k_2},...,b_n^{\pm 1}\}$ which represent the same element in $G$. We may assume that $w'$ is shortlex and $w$ is not.
 
 	 Therefore, by Theorem \ref{NF2} $w$ should admit a critical reducing sequence, so it must have a critical subword. But by Definition \ref{critical words} it is impossible to have a critical subword on $b_i^{\pm k_i}, b_j^{\pm k_j}$ if $k_i,k_j\geq 2$ and $m_{b_i,b_j}\geq 4$. Thus $w$ doesn't admit a critical reducing sequence, and $w$ must be shortlex.
 	 Therefore, $w=w'$.
 \end{proof}
 
 \begin{lem}\label{DeltaDifferent}
 Assume $g\in\Omega_i^{\pm}\cap \Omega_j^{\pm}$ and $\hat{w}_i\in \delta(\rho_i^{\pm}(g))$ and $\hat{w}_j\in \delta(\rho_j^{\pm}(g))$, then $\hat{w}_i\neq_G\hat{w}_j$.
 \end{lem}
 
 \begin{proof}
 	Note that $\hat{w}_i,\hat{w}_j$ both lie in the coset $Hg$. Notice as well that $\delta(\rho_i^{\pm}(g)), \delta(\rho_j^{\pm}(g)) \subset \delta(g)$. Since $\hat{w}_i\in\delta(g)$, then $\hat{w}_i=_G h_ig_i$ ($h_i\in H$  beginning with $b_i^{\pm}$). Analogously, $\hat{w}_j=_G h_jg_j$ ($h_j\in H$ beginning with $b_j^{\pm})$. Then, $\hat{w}_i=_G\hat{w}_j$ implies that $h_i=_G h_j$, but this is impossible since $H$ is free by Lemma \ref{subgroupH} and $h_i, h_j$ begin with different letters.
 \end{proof}
 
 \begin{lem}\label{lema length}
 	The minimal geodesic length of an element $g\in \Omega_i^{\pm}\cap \Omega_j^{\pm}$, $i\neq j$ is bounded below by $|n_j^-|+|n_i^-|(2m-1)$ where $|n_j^-|\geq|n_i^-|$ and $m=m_{b_i,b_j}$. Moreover, there is always an element $g$ of that geodesic length in the set $\Omega_i^{-}\cap \Omega_j^{-}$.
 \end{lem}
 
 \begin{proof}
 	Assume first that $g\in \Omega_i^+\cap\Omega_j^-$. Then, by Lemma \ref{obs2}, $n_j^-=-1$ thus also $n_i^-=-1$. Using Lemma \ref{Lema 9} we have that the geodesic length of $g$ is at least \begin{equation}\label{Eq0}|p_i^+|+(2m-1)
 	> 1+(2m-1)=|n_j^-|+|n_i^-|(2m-1)
 	\end{equation} where the inequality is strict since $p_j^+\geq 2$ by Remark \ref{obs1} (1).
 	
 	 Now assume that $g\in \Omega_i^-\cap\Omega_j^+$, by Lemma \ref{Lema 8} and Remark \ref{obs1} (1) $n_i^-=-1$. By Lemma \ref{Lema 9} the minimal geodesic length of such an element is \begin{equation}\label{EqA}
 	 |p_j^+|+(2m-1)\geq |n_j^-|+(2m-1)=|n_j^-|+|n_i^-|(2m-1)\end{equation} where by Lemma \ref{numeritos} the equality holds if and only if $k_j$ is odd.
 	
 	If $g\in \Omega_i^+\cap\Omega_j^+$, by Lemma \ref{Lema 7} the minimal geodesic length of $g$ is \begin{equation}
 	\label{EqB} |p_j^+|+|p_i^+|(2m-1)\geq |n_j^-|+|n_i^-|(2m-1) 
 	\end{equation} where by Lemma \ref{numeritos} the equality holds if and only if $k_i, k_j$ are odd.
 	
 	Finally, if $g\in \Omega_i^-\cap\Omega_j^-$, by Lemma \ref{Lema 7} the minimal geodesic length of $g$ is $$|n_j^-|+|n_i^-|(2m-1)$$ and also by Lemma \ref{Lema 7} we know that there exists an element of this length satisfying  $g\in\Omega_i^-\cap\Omega_j^-$.
 	\end{proof}
 
 \begin{obs}\label{FirstStepInduction}
 	Let $t=n_i^-$, $s=n_j^-$, $b=v_i, a=v_j$.
 	If $k_j$ is even, in Lemma \ref{lema length} the inequalities (\ref{Eq0}), (\ref{EqA}) and (\ref{EqB}) are strict and therefore there is only one element $g\in \Omega_i^{\pm}\cap \Omega_j^{\pm}$, $i\neq j$ with geodesic length $|n_j^-|+|n_i^-|(2m-1)$. This element is: $$b^t[_{2m-1}(a^{-1},b^{-1})]^{|t|}a^{-|s-t|}=_G  a^s[_{2m-1}(b^{-1},a^{-1})]^{|t|}$$
 	
 	If $k_j$ is odd and $k_i$ is even, then $k_i\geq 2$ and therefore by Lemma \ref{obs2b} $\Omega_i^+\cap \Omega_j^-=\Omega_i^-\cap \Omega_j^+=\emptyset$. The inequality (\ref{EqB}) is strict. Therefore there is again only one element $g\in \Omega_i^{\pm}\cap \Omega_j^{\pm}$ of geodesic length $|n_j^-|+|n_i^-|(2m-1)$. This element is
 	$$b^t[_{2m-1}(a^{-1},b^{-1})]^{|t|}a^{-|s-t|}=_G a^s[_{2m-1}(b^{-1},a^{-1})]^{|t|}$$

 	If $k_i, k_j$ are both odd numbers, applying Lemmas \ref{Lema 7}, \ref{Lema 8} and \ref{Lema 9}, we obtain that there are two elements of that geodesic length in $\Omega_i^{\pm}\cap \Omega_j^{\pm}$:
 		$$b^t[_{2m-1}(a^{-1},b^{-1})]^{|t|}a^{-|s-t|}=_G a^s[_{2m-1}(b^{-1},a^{-1})]^{|t|},$$
 		$$b^{-t}[_{2m-1}(a,b)]^{|t|}a^{|s-t|}=_G a^{-s}[_{2m-1}(b,a)]^{|t|}.$$

 	In this way, if we consider the order $a<a^{-1}<b<b^{-1}<...$ we have that the shortlex representative in $\Omega_i^{\pm}\cap \Omega_j^{\pm}$ is :
 	
 	\begin{itemize}
 		\item If $k_i, k_j$ are both odd numbers, $$b^{-t}[_{2m-1}(a,b)]^{|t|}a^{|s-t|}=_G a^{-s}[_{2m-1}(b,a)]^{|t|}$$ (which is a positive word).
 		
 		\item Otherwise, $$b^t[_{2m-1}(a^{-1},b^{-1})]^{|t|}a^{-|s-t|}=_G a^s[_{2m-1}(b^{-1},a^{-1})]^{|t|}$$ (which is a negative word).
 	\end{itemize}
 	
 \end{obs}
 
 Recall that we are considering the group:
 $$K:=\langle\langle r\rangle \rangle=\langle r^{g}; g\in G_1 \mid \hat{R}\rangle $$
 
 with 
 \begin{align}
 &\hat{R}^{+}=\left\{r^{b_{i}^{p_i^+}g}=r^{b_{i}^{p_i^+-1}g}...r^{b_i^{p_i^-}g}...\left(r^{b_{i}^{p{i}-1}g}\right)^{-1};  b_i^{p_i^+}g\in\Omega_i^{+} \text{ for some } 1\leq i \leq n\right\} \nonumber\\
 &\hat{R}^{-}=\left\{r^{b_{i}^{n_i^-}g}=\left(r^{b_i^{n_i^-+1}g}\right)^{-1}...r^{b_i^{n_i^+}g}...r^{b_i^{n_i^-+1}g}; b_i^{n_i^-}g\in\Omega_i^{-} \text{ for some } 1\leq i \leq n\right\} \nonumber\\
 & \hat{R}=\hat{R}^{+}\cup \hat{R}^{-} \nonumber                     
 \end{align}
 
 Notice that the relations of $\hat{R}^+$ and $\hat{R}^-$ have the following form:
	\begin{align}
	&\hat{R}^{+}=\{r^{h}=\alpha r^{\rho_i^+(h)}\alpha^{-1};  h\in\Omega_i^{+} \text{ for some } 1\leq i \leq n\} \nonumber\\
	&\hat{R}^{-}=\{r^{h}=\alpha r^{\rho_i^{-}(h)}\alpha^{-1}; h\in\Omega_i^{-} \text{ for some } 1\leq i \leq n\} \nonumber\\
	& \hat{R}=\hat{R}^{+}\cup \hat{R}^{-} \nonumber                     
	\end{align}

\begin{prop}\label{freeness}
	$K$ is a free group.
\end{prop}

\begin{proof}
	We are going to prove it using Tietze transformations. We order the vertices  $b_1,...,b_n$ linked to $r$, in such way that $k_1\leq k_2 \leq ...\leq k_n$ and we consider the lexicographic order $$b_1<b_1^{-1}<b_2<b_2^{-1}<...<b_n<b_n^{-1}<c_1<c_1^{-1}<...<c_k<c_k^{-1}.$$ Notice that the choice of this order is  consistent with the results of Section \ref{Section3}, since there we only used the chosen orders as a technical tool to prove the possible existence or not of determinate kinds of geodesic representatives, but the results themselves didn't depend on the chosen order.
	
		Firstly, let us consider those generators of the form $r^{g}$, such that $g\in \Omega$. For such an element there is at least one relation in $R$ which is of the form $r^{g}=\alpha r^{\rho_i^{\pm}(g)}\alpha^{-1}$ (see Remark \ref{obs5}), it will belong to $R^{+}$ or $R^{-}$ depending on whether $g$ has a geodesic representative beginning with $b_i^{p_i^+}$ or $b_{i}^{n_i^-}$ respectively ($i=1,2,...,n$). Thus, using Tietze transformations we can erase that relation and the generator $r^{g}$. This can be done to erase every generator $r^{g}$ with $g\in \Omega$, but maybe not every relation of $R$. Because if $g$ lies in two of the sets $\Omega_i^{\pm}$, $\Omega_j^{\pm}$ (remind that by Lemma \ref{Max3}, any element can belong to at most two sets $\Omega_i^{\pm}$), then after using Tietze transformations we are left with a relation of the following form: 
	\begin{equation}\label{Form}
\alpha^{-1} r^{\rho_i^{s}(w_1)}\alpha=\beta^{-1} r^{\rho_j^{t}(w_2)}\beta.
	\end{equation}

	We define $\tilde{R}$ as the set of remaining relations after this process. Recall that $\Lambda=G_1\setminus \Omega$.

	After this process we get a presentation:
	$$K=\langle r^{g}; g\in G_1, g\in\Lambda \mid \tilde{R}\rangle.$$

	Now we will see that we can also get rid of the relations in $\tilde{R}$. Notice that $$\Omega:=\bigcup\left\{ \Omega_i^{\pm}\cap\Omega_j^{\pm}\mid i,j=1,...,n, i\neq j \right\}\subset\left\{g\in G_1\mid \#(Hg\cap\Lambda)>1\right\}$$
	since the elements in $\Omega$ are of the form $b_{j_1}^{\epsilon_1 k_{j_1}}...b_{j_n}^{\epsilon_n k_{j_n}}g$ and if two expressions represent the same group element, then by Lemma \ref{DeltaDifferent} the $b_{j_i}'s$ must also be equal. Notice that there is a bijection between the relations in $\hat{R}$ and the elements of $\Omega$ since each relation appears when we have an element in $\Omega$.  Also recall that for any $g\in\Omega$ we have $\# \rho(sl(g))=2$ by Lemma \ref{obs3} (ii). 
	
	Let us explain a little the strategy that we are going to follow. At this point we have a presentation of the group $K$ with:
	
	\begin{enumerate}
		\item Set of generators $\{r^{g}, g\in\Lambda\}$, bijective to $\Lambda$.
		
		\item Set of relators $\tilde{R}=\{rel(h)\mid h\in\Omega\}$, bijective to $\Omega$.
	\end{enumerate}
	
	We want to show that it is possible to remove all the relators and some generators using Tietze transformations. To do that we proceed inductively. More precisely, we first order the elements of $\Omega$ using the shortlex order:	
	$$h_1<h_2<...<h_i<...$$
	
	To each element $h_i$ we associate a $u_i\in \Lambda$ such  that the relator $rel(h_i)$ can be written as:
	$$r^{u_i}= \gamma r^{w_i} \gamma^{-1}$$
	where $\gamma$ is a word in the alphabet $\{r^{g}, g\in \Lambda -\{u_1,...,u_i\}\}$.

	Notice that by the order that we have established in the vertices, we need to take $h_1\in \Omega_1^{\pm}\cap \Omega_2^{\pm}$. By Remark \ref{FirstStepInduction} for the element $h_1$ we have the following possibilities:
	
	\begin{itemize}
		\item If $k_1, k_2$ are both odd numbers, then $h_1$ admits the following two geodesic representatives:
		
		$w_1=b_1^{p_1^+}$$ _{2m-1}(b_2,b_1)^{p_2^+}b_2^{p_2^+-p_1^+}$, $w_2=b_2^{p_2^+} $$_{2m-1}(b_1,b_2)^{p_2^+}$.
		
		\item Otherwise, $h_1$ admits the following two geodesics representatives: $w_1=b_1^{n_1^-}$$ _{2m-1}(b_2^{-1},b_1^{-1})^{|n_2^-|}b_2^{n_2^--n_1^-}$, $w_2=b_2^{n_2^-} $$_{2m-1}(b_1^{-1},b_2^{-1})^{|n_2^-|}$.
	\end{itemize}
	 Where in both cases $m=m_{12}$. 
		
		 We know that $|\rho_s^-(w_s)|=|w_s|$ if and only if $k_s$ is even and that $|\rho_s^+(w_s)|<|w_s|$. Then, we can distinguish the following cases:
		
		\begin{enumerate}
			\item If $k_1, k_2$ are odd numbers, we have that for $s=1,2$,  $|\rho_s^+(w_s)|<|w_s|$, which implies by minimality that $\rho_s^+(w_s)\not\in \Omega$. Thus $\# \delta (\rho_s^+(w_s))=1$, so $\delta(h_1)=\{\delta(\rho_1^+(w_1)),\delta(\rho_2^+(w_2))\}$ has only two elements.
			
			\item If $k_2$ is even and $k_1$ is odd (or the other way around), then, $|\rho_1^-(w_1)|<|w_1|$, so $\rho_1^-(w_1)\not\in\Omega$ and $\#\delta(\rho_i^-(w_i))=1$. And $|\rho_2^-(w_2)|=|w_2|$ and by, Lemma \ref{Lema 10}, $\rho_2^-(w_2)\not\in\Omega$ and, since by Lemma \ref{LexCond} $sl(\rho_2^-(w_2))<_{slex} sl(w_2)$, and $w_1=_G w_2$ are the shortlex minimal words in $\Omega$, we have that $\# \delta(\rho_2^-(w_2))=1$. So, $\delta(h_1)=\{\delta(\rho_1^-(w_1)),\delta(\rho_2^-(w_2))\}$ has only two elements.
			
			\item If $k_1, k_2$ are both even, then for $s=1,2$ $|\rho_s^-(w_s)|=|w_s|$ and, by Lemma \ref{Lema 10}, $\rho_s^-(w_s)\not\in\Omega$ and again since by Lemma \ref{LexCond} $sl(\rho_s^-(w_s))<_{slex} sl(w_s)$ and $w_1=_G w_2$  are the shortlex minimal words in $\Omega$, we have that $\# \delta(\rho_s^-(w_s))=1$. So, $\delta(h_1)=\{\delta(\rho_1^-(w_1)),\delta(\rho_2^-(w_2))\}$ has only two elements.
		\end{enumerate}
		
		So, in every possible case $\delta(h_1)=\{\hat{w}_1,\hat{w}_2\}\subset \Lambda$ and both lie in the coset $Hh_1$ in $\Lambda$. Besides, $\hat{w}_1\neq_G \hat{w}_2$ by Lemma \ref{DeltaDifferent}.
		
		Therefore, we may assume $sl(\hat{w}_2)>_{slex} sl(\hat{w}_1)$ and using a Tietze transformation we can erase the relation and the generator $r^{\hat{w}_2}$.
		
We define the set $\Lambda_{1}=\Lambda\setminus \{\hat{w}_2\}$. Now, we can define the natural projection $\pi_1:\mathcal{P}(\Lambda) \rightarrow \mathcal{P}(\Lambda_1)$ given by $\pi_1(A)=A\cap \Lambda_1$. After this, we define $\delta_1=\pi_1 \circ \delta: \Omega\longrightarrow \mathcal{P}(\Lambda_1)$.

		We are going to prove that we can construct a family of subsets $\Lambda=\Lambda_0\supset\Lambda_1 \supset \Lambda_2 \supset...$ such that for each element $g$ of $G_{1}$ if $sl(g)<_{slex} sl(h_{k})$, then $\#(\Lambda_{l}\cap \delta(g))= 1$ for every $l \geq k-1$. Once we have constructed these sets, we can define the applications $\pi_k:\mathcal{P}(\Lambda) \rightarrow \mathcal{P}(\Lambda_k)$ and $\delta_k=\pi_k \circ \delta: \Omega\longrightarrow \mathcal{P}(\Lambda_k)$.
	\begin{enumerate}
		
		\item For $k=1$, take $\Lambda_1$.

		\item Assume $\Lambda_1,...,\Lambda_{k-1}$ have been constructed. For $k$ we know that $\#(\Lambda_{k-1}\cap \delta(g))=1$ for every $g$ such that $sl(g)<_{slex} sl(h_{k})$. Also $h_k\in \Omega_i^{\pm}\cap \Omega_j^{\pm}$ for some $i,j\in\{1,...,n\}$ with $i\neq j$. At a first step we have two geodesic representatives of $h_k$: $w_i$ beginning with $b_i^{p_i^+}$ or $b_i^{n_i^-}$ and $w_j$ beginning with $b_j^{p_j^+}$ or $b_j^{n_j^-}$.
		
		By Lemma \ref{Proposition 12} for each $s=i,j$ there exist $l_s\in \mathbb{Z}$ such that $|\rho^{l_s}(w_s)|<|w_s|$ (and $|\rho^{d_s}(w_s)|=|w_s|$ for $d_s<l_s$). Note that by Lemma \ref{Lema 10} each of the elements represented by $w_s$, $\rho(w_s)$,..., $\rho^{l_s-1}(w_s)$ lies at most on one of the sets $\Omega_r^{\pm}$, $r\in\{1,...,n\}$.
		
		Therefore, $sl(\rho^{l_s}(w_s))<_{slex} sl(h_k)$ and by construction of $\Lambda_{k-1}$, $\Lambda_{k-1}\cap (\rho^{l_i}(w_s))=\tilde{w}_s$. Therefore $\delta_{k-1}(h_k)\cap \Lambda_{k-1}=\{\tilde{w}_i, \tilde{w}_j\}$ (and both must be different by Lemma \ref{DeltaDifferent}).
		
		Now, we may assume $sl(\tilde{w}_j)>_{slex} sl(\tilde{w}_i)$, 
		and we define $\Lambda_k=\Lambda_{k-1}\setminus \{\tilde{w_j}\}$.

	\end{enumerate}
	
	Now, at each step of the induction, when we obtain $\delta_{k-1}(h_k)\cap \Lambda_{k-1}=\{\tilde{w}_i, \tilde{w}_j\}$, we know that the relation $rel(h_k)$ can be written as $r^{\tilde{w}_j}$ equals to a conjugate of $r^{\tilde{w}_i}$. And using a Tietze transformation we may eliminate the relation and the generator $r^{\tilde{w}_j}$.
	
	Notice that by the construction of the family $\{\Lambda_k\}$, the word $\hat{w}_1$ obtained in the construction of $\Lambda_1$ verifies that $\hat{w}_1\in\Lambda_k$ for any $k\in\mathbb{N}$. Therefore, notice that $\cap_{k\in\mathbb{Z}} \Lambda_k \neq \emptyset$.
	In this way, after each inductive step we have a presentation of the group $K$ with:
	
	\begin{enumerate}
		\item Set of generators $\{r^{g}, g\in\Lambda_k\}$, bijective to $\Lambda_k$.
		
		\item Set of relators $\tilde{R}_k=\{rel(h_i)\mid i=k+1,...\}$, bijective to $\Omega\setminus \{h_1,...,h_k\}$.
	\end{enumerate}
	
	Thus, eventually we can remove every relation and we conserve a non-empty set of generators $\cap_{k\in\mathbb{Z}} \Lambda_k$. Therefore, the group is free.
	\end{proof}

And now, as an inmediate consequence we obtain our main theorem:

\begin{teo}\label{Mainthm}
	Any even Artin group based on a large graph is poly-free.
\end{teo}

\begin{proof}
	It follows immediately from Lemma \ref{action} and Proposition \ref{freeness}.
\end{proof}

 \section{Poly-freeness for even Artin groups based on triangle graphs}\label{S6}
 
Consider the case when our graph is a triangle. We can distinguish four different types of triangles according to the number of edges with label $2$:

\begin{enumerate}[(i)]
	\item $(2,2,2)$,
	
	\item $(2k_1,2,2)$,
	
	\item $(2k_1,2k_2,2),$
	
	\item $(2k_1,2k_2,2k_3)$.
\end{enumerate}     with $k_i\geq 2$.
 
 The Artin group associated to a triangle of type (i) is $\mathbb{Z}^3$, so it is poly-free. Artin groups associated to type (ii) triangles are even of type FC and so we know that it is also poly-free by \cite{Blasco1} (In fact, these groups are of the form $A_2(2k_1)\times \mathbb{Z}$ so they are obviously poly-free). And Artin groups associated to triangles of type (iv) are large even Artin groups and thus they are also poly-free by Theorem \ref{Mainthm}. The only remaining case are Artin groups associated to triangles of type (iii).

 
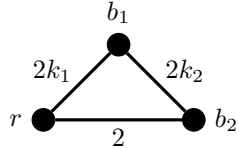
\begin{figure}[ht]
 	
 	\begin{center}
 		
 		\begin{tikzpicture}[vertice/.style={draw,circle,fill,minimum size=0.3cm,inner sep=0}]
 		\node[vertice, label=right: $b_2$] (D1) at (1,.5) {};
 		\node[vertice, label=above: $b_1$] (E1) at (0,1.5) {};
 		\node[vertice, label=left: $r$] (F1) at (-1,.5) {};
 		\draw[very thick] (D1)--(F1)--(E1)--(D1);
 		\node[left] at (-0.5,1.15) {$2k_1$};
 		\node[right] at (0.5,1.15) {$2k_2$};
 		\node[below] at (0,.5) {$2$};
 		\end{tikzpicture}
 	\end{center} 	
 	
 \end{figure}

 The problem for this case is that as far as we know there are not known normal forms for the associated group. 
 
 However, notice that in every proof along the paper we have only used normal forms in the small subgroup $A_{\Gamma\setminus\{r\}}$, never in the big Artin group $A_{\Gamma}$. Thus, almost the exact same proof that we have used to prove poly-freeness for large even Artin groups works also for any even Artin group $\Gamma$ satisfying that there exists a vertex $r\in V(\Gamma)$ such that $A_{\Gamma\setminus\{r\}}$ is a large even Artin group.
 
 The only place where we need to change a bit our proof is in Lemma \ref{action} since Lemma \ref{Commute} is not true for this Artin group. In this way, in $G_1\ltimes K$ we could not have discarded the relations $[Stab(r),r]$. We are going to give a different proof for Lemma \ref{action} in this situation.

\begin{lem}\label{action2}
	
	Let $A_{\Gamma}=\langle r, b_1, b_2 \mid (rb_1)^{k_1}=(b_1r)^{k1}, rb_2=b_2r, (b_1b_2)^{k_2}=(b_2b_1)^{k_2}\rangle$ and $A_{\Gamma_1}=A_{\Gamma\setminus \{r\}}=\langle b_1, b_2 \mid (b_1b_2)^{k_2}=(b_2b_1)^{k_2}\rangle$. Let us consider the map: $$\psi: A_{\Gamma} \longrightarrow A_{\Gamma_1},$$ induced by $r\mapsto 1$ and $b_i \mapsto b_i$ for $i=1,2$.
	
	$\ker(\psi)$ is isomorphic to:
	$$K:= \langle r^{g}; g\in A_{\Gamma_1} \mid \hat{R}\rangle,$$ 
	where $\hat{R}$ is the set of relations defined in Definition \ref{Relaciones}.
\end{lem}

\begin{proof}
	Note that $\ker(\psi)$ is the normal subgroup of $A_{\Gamma}$ generated by $r$. 
	 
	 Notice that in this case we have $p_1^+=2, p_1^-=0, n_1^-=-1, n_1^+=1, p_2^+=1, p_2^-=0, n_2^-=-1, n_2^+=0$. Hence, we have $\hat{R}=\hat{R}^+\cup \hat{R}^-$ with: $$R^+=\{r^{b_1^2g}=r^{b_1g}r^{g}\left(r^{b_1^{-1}g}\right)^{-1}; b_1^2g\in \Omega_1^+\}\cup\{r^{b_2g}=r^{g}; b_2g\in\Omega_2^+\}$$
	 
	 $$R^-=\{r^{b_1^{-1}g}=\left(r^{g}\right)^{-1}r^{b_1g}r^{g}; b_1^{-1}g\in \Omega_1^-\}\cup\{r^{b_2^{-1}g}=r^{g}; b_2^{-1}g\in\Omega_2^-\}$$
		As before, we define an action of $G_1=A_{\Gamma_1}=A_{\Gamma\setminus \{r\}}=\langle b_1, b_2 \mid (b_1b_2)^{k_2}=(b_2b_1)^{k_2}\rangle$ on the abstract group $K$ via
	$$h^{-1}(r^{g})h=r^{gh}, \hspace{0.3 cm} h\in G_1.$$
	The proof is exactly the same as in Lemma \ref{action} until we obtain the presentation of the semidirect product applying  Lemma \ref{subgroup}. In our case, we obtain
	$$G_1\ltimes K =\langle r, b_1, b_2 \mid \hat{C}_0, T, [Stab_{G_1}(r),r] \rangle.$$
	 where $\hat{C}_0=\hat{C}^+\cup \hat{C}^-$, $T=\{(b_1b2)^{k_2}=(b_2b_1^{k_2})\}$ with
	\begin{align*}
	&\hat{C}^{+}=\{(b_{1}^{2}g)^{-1}r (b_{1}^{2}g)=(b_{1}g)^{-1}r (b_{1}g)(g)^{-1} r (g)(b_{1}g)^{-1}r (b_{1}g))^{-1}; b_1^2g\in\Omega_1^+\}\cup \\ &\hspace{1.2 cm} \{(b_2g)^{-1}r(b_2g)=(g)^{-1}r(g); b_2g\in\Omega_2^{+} \} \nonumber\\
	&\hat{C}^-=\{(b_{1}^{-1}g)^{-1}r (b_{1}^{-1}g)=((g)^{-1}r (g))^{-1}(b_1g)^{-1}r (b_1g)(g)^{-1}r (g); b_1^{-1}g\in\Omega_1^{-}\}\cup\nonumber \\&\hspace{1.2 cm} \{(b_2^{-1}g)^{-1}r(b_2^{-1}g)=(g)^{-1}r(g); b_2^{-1}g\in\Omega_2^{-} \} \nonumber\\              
	\end{align*}

	We define:
	$$C'_0=\{b_1^{-2}rb_1^{2}=b_1^{-1}rb_1r(b_1^{-1}rb_1)^{-1},b_2^{-1}rb_2=r\}$$

	So in fact, the relations in $\hat{C}_0\setminus C'_{0}$ are obtained from the ones of $C'_{0}$ by conjugation. So we can eliminate them from the presentation using Tietze transformations. Thus, we have:
	$$G_1\ltimes K=\langle b_1, b_2 ,r\mid C'_0, T, [Stab_{G_1}(r),r]\rangle.$$
	
	Note that $\{C_0', T\}$ is in fact the set of relations in $A_{\Gamma}$. Therefore, we have an epimorphism $$A_{\Gamma} \twoheadrightarrow  G_1\ltimes K.$$
	
	To end the proof it is enough to see that the relations $[Stab_{G_1}(r),r]$ are also satisfied in $A_{\Gamma}$.
	
	Let us consider $g\in Stab_{G_1}(r)$, therefore by the definition of our action $r^{g}=_{K}r$. Now, take into account the following property of the set of relations $\hat{R}$. By construction, each time that we have a relation $\hat{R}^+$:
	$$r^{b_{i}^{p_i^+}g}=r^{b_{i}^{p_i^+-1}g}...r^{b_i^{p_i^-}g}...\left(r^{b_{i}^{p{i}-1}g}\right)^{-1},$$
	
	then in the original $A_{\Gamma}$ the same relation it is also satisfied. 
	
	Analogously, we have a similar situation for the relations of $\hat{R}^-$.
	
	 Therefore, $r^{g}=_K r$ implies that $r^g=_{A_{\Gamma}} r$, i.e. $g^{-1}rg=_{A_{\Gamma}} r$. Thus, $g^{-1}rg=_{A_{\Gamma}}r$ for every $g\in Stab_{G_1}(r)$.
	 
	  But, $[Stab_{G_1}(r),r]=\{g^{-1}rg=r \mid g\in Stab_{G_1}(r) \}$, so the relations $[Stab_{G_1}(r),r]$ are also satisfied in $A_{\Gamma}$.
	
Thus, we have:
$$G_1\ltimes K=\langle b_1, b_2, r\mid C'_1, T\rangle \simeq A_{\Gamma}$$ and the isomorphism maps $K$ onto $\ker(\psi)$.	Therefore, $K  :=\langle r^{sl(g)}; g\in G_1 \mid \hat{R}\rangle= \langle\langle r\rangle\rangle= \ker(\psi)$. 
\end{proof}
 
 \begin{obs}
 	Notice that this proof could have also been applied to the case of large Artin groups proved before. But we have preferred to present the proofs in this way because we think that Lemma \ref{Commute} has importance by itself.
 \end{obs}

As we have commented before, the rest of the proof for polyfreeness works in almost the same way as for large Artin groups. The proof of Proposition \ref{freeness} can be used also in this case just taking into account the following remarks:

\begin{itemize}
	\item The proof of Lemma \ref{obs2b} cannot be used in this case, but in our particular case we have $n_1^-=n_2^-=-1$, so we don't need to apply it.
	
	\item It is not necessary to use Remark \ref{FirstStepInduction} in order to check different cases. Since we are considering a particular case, we know that for this group $h_1=a^{-1}b^{-1}a^{-1}b^{-1}=_G b^{-1}a^{-1}b^{-1}a^{-1}$.
\end{itemize}

 Therefore:
 
 \begin{corol}
 	The Artin group based on the triangle graph $(2k_1,2k_2,2)$ is poly-free.
 \end{corol}
 
 \begin{corol}
 	Any even Artin group based on a triangle graph is poly-free.
 \end{corol}

\bibliographystyle{acm}
\bibliography{Poly-freeness_in_large_even_Artin_groups._R._Blasco}

\end{document}